\documentclass[french,11pt,a4paper]{article}
\usepackage[utf8]{inputenc}
\usepackage[T1]{fontenc}
\usepackage[sectionbib]{natbib}
\usepackage[english]{babel}
\usepackage{float}
\usepackage{anysize}
\usepackage{amsmath}
\usepackage{amsthm,amsfonts,amssymb}
\usepackage{xcolor}
\usepackage[overload]{empheq}
\usepackage{fullpage}
\usepackage{subcaption}
\usepackage{authblk}
\usepackage{cancel}
\usepackage{dsfont}
\usepackage[section]{placeins}
\usepackage{csvsimple}
\usepackage{bm}
\usepackage{ulem}
\usepackage{enumitem}
\numberwithin{equation}{section}

\usepackage{caption}
\usepackage{xcolor} % où color selon l'installation
\usepackage{mdframed}
\usepackage{multirow} %% Pour mettre un texte sur plusieurs rangées
\usepackage{multicol} %% Pour mettre un texte sur plusieurs colonnes
\usepackage{tikz,tkz-tab}
\usepackage{graphicx}
\usepackage[absolute]{textpos} 
\usepackage{colortbl}
\usepackage{array}
\usepackage[top=4cm, bottom=4cm, left=2.5cm, right=2.5cm]{geometry}
\usepackage{titlesec}
\usepackage[english]{cleveref}

\DeclareMathOperator{\opdiv}{div}

\newcommand{\bx}{\mathbf{x}}

\newcommand{\bw}{\mathbf{w}}

\newcommand{\bu}{\mathbf{u}}

\newcommand{\bn}{\mathbf{n}}

\newcommand{\dx}{\, \mathrm{d}}

\def\xR{\mathbb{R}}

\newtheorem{theo}{Theorem}[section]

\newtheorem{defi}[theo]{Definition}
\newtheorem{prop}[theo]{Proposition}
\newtheorem{proposition}[theo]{Proposition}

\newtheorem{lemma}[theo]{Lemma}

\newtheorem{rema}[theo]{Remark}

\makeatletter
\def\mathcolor#1#{\@mathcolor{#1}}
\def\@mathcolor#1#2#3{%
	\protect\leavevmode
	\begingroup
	\color#1{#2}#3%
	\endgroup
}
\makeatother

\usepackage{import}
\usepackage{color}

\title{Existence of traveling wave for a coupled incompressible Darcy’s free boundary model with undercooling effect and surface tension}
\author[1]{Claire Alamichel}
\author[2]{Nicolas Meunier}
\affil[1]{Univ. Rennes, Inria, IRMAR - UMR 6625, F-35000 Rennes, France. (claire.alamichel@inria.fr)}
\affil[2]{LaMME, UMR 8071 CNRS, Universit\'e \'Evry Val d'Essonne, France. (nicolas.meunier@univ-evry.fr)}

\begin{document}
	
	\maketitle
	
	\begin{center}
    %\today
    
    %\vspace*{1cm}
    
    {\large \textbf{Abstract}}
    \end{center}
    
    In this paper, we present a cell motility model that takes into account the cell membrane effect. The model introduced is an incompressible Darcy free boundary problem. This model involves a nonlinear term in the boundary condition to model the action of the membrane. This term can be seen as a undercooling effect of the membrane on the cell. It also implies a destabilizing nonlinear term in the boundary condition, depending on polarity markers and modeling the active character of the cytoskeleton. First, we study the linear stability of the steady state and prove that above a threshold, the disk is linearly unstable. This analysis highlights the stabilizing effect of undercooling. Then, using a bifurcation argument, we prove the existence of traveling waves that describe a persistent motion in cell migration and justify the relevance of the model.
    
    {%\hypersetup{linkcolor=black}
    	\tableofcontents
    }
    
    \medskip

	\section{Introduction}
	\subsection{Biological context and modeling}
	
	Cell motility is a biological process involved in numerous biological phenomena such as immune response, wound healing, embryonic development and cancer propagation. Although different types of cell migration exist, they are all governed by the same mechanisms and rely on a cell's ability to polarize. This means that the cell is able of breaking its internal symmetry and possessing a defined front and rear \citep{danuser_mathematical_2013,reig_cell_2014}. Once the cell is polarized, cell migration can occur. This relies on cytoskeletal activity via the polymerization of actin filaments at the front of the cell, and the contraction of acto-myosin fibers that push the cytoskeleton from rear to front \citep{abercrombie_croonian_1980}. 
	
	The cell is enclosed by the plasma membrane which protects and separates the contents of the cell from the external environment. It also enables communication with the external environment. The plasma membrane contributes to the mechanical properties of cell movement via membrane tension. Membrane tension induces an opposite force to membrane extension. It results from the inextensible nature of the membrane, which creates in-plane tension, and from the energy derived from adhesion between the cytoskeleton and the membrane. As the cell is deformed, membrane tension changes rapidly, causing the cell surface tension to vary. Membrane tension also regulates polarisation between the front and rear of the cell \citep{thiam_cell_nodate}.
	
	In this paper, we propose and study a model describing this phenomenon, taking into account the action of the membrane on cell motility. The proposed model is a free-boundary model given by:
    \begin{subequations}\label{eq:model}
		\begin{align}[left=\empheqlbrace]
			& \bu + \nabla P = 0 && \text{in } \Omega\left(t\right), \label{eq:model_darcy} \\
			& \opdiv\left( \bu\right) = 0 && \text{in } \Omega\left(t\right), \label{eq:model_incompressibility} \\
			& V_n = \bu \cdot \mathbf{n} && \text{on } \partial \Omega \left(t\right), \label{eq:model_normal_velocity}\\
			& P = \gamma \kappa + \chi_c f_{\mathrm{act}}\left(c\right) + \chi_u f_{\mathrm{und}} \left(V_n\right) && \text{on } \partial \Omega \left(t\right), \label{eq:model_pressure_boundary} \\
			& \partial_t c = \opdiv \left( \nabla c - \left(1-a\right) \bu c\right) && \text{in }  \Omega\left(t\right),  \label{eq:model_markers_dynamic}\\
			& \left( \nabla c + a \bu c\right)\cdot \bn = 0 && \text{on }  \partial \Omega\left(t\right) , \label{eq:model_markers_boundary}\\
			& c(0,\bx) = c^{\mathrm{in}}(\bx) && \text{in } \Omega^{\mathrm{in}}, \label{eq:model_initial_condition} \\
			& \Omega\left(0\right) = \Omega^{\mathrm{in}}.&& \label{eq:model_initial_domain}
		\end{align}
	\end{subequations}
	
	This free-boundary model is in continuity with the model proposed by \citet{lavi_motility_2020}. The cell is modeled by a droplet of incompressible fluid confined between two parallel plates and containing polarity markers of concentration $c$. $\bu$ denotes the gap-averaged planar flow and $P$ the pressure of the fluid. Thin-film lubrication and Hele-Shaw approximations yield velocity $\bu$ and pressure $P$ satisfying Darcy's law \cref{eq:model_darcy}. The kinematic condition \eqref{eq:model_normal_velocity} indicates that the normal velocity of the sharp interface is given by the normal velocity of the fluid.
	
	Initially, the fluid is contained in $\Omega^{\mathrm{in}} \subset \mathbb{R}^2$ a simply connected bounded open. $c^{\mathrm{in}}$ is a given smooth non-negative function defined on $\Omega^{\mathrm{in}}$, which represents the concentration of a solute at time $t = 0$. In \eqref{eq:model}, we seek a family of open sets $\Omega(t)$ of $\xR ^2$ with boundary $\partial \Omega(t)$, whose curvature (positive for a circle) is $\kappa$, and a concentration function $c\left(t,\bx\right)$ defined on $\Omega\left(t\right)$. 
	
	The cell membrane is modeled by the boundary $\partial \Omega\left(t\right)$. The action of the membrane on the cell is modelled by the force $-\chi_u f_{\mathrm{und}}\left(V_n\right)$ where $\chi_u>0$ and $V_n$ is the normal velocity of the boundary. We assume that $f_{\mathrm{und}}$ satisfies:
	\begin{subequations}\label{eq:conditions_f_und}
		\begin{align}
			& f_{\mathrm{und}} \in C^1\left(\mathbb{R}\right), \\
			& f_{\mathrm{und}} \text{ is an odd and increasing function}, \\
			& f'_{\mathrm{und}}\left(0\right) > 0. %\\
			%& \lim\limits_{x\rightarrow + \infty} f_{\mathrm{act}}\left(x\right) = L_c < + \infty.
		\end{align}
	\end{subequations}
	The case $\chi_u = 0$ means neglecting the action of the cell membrane and considering that it behaves in the same way as the rest of the cell. In this case, we find the model of \citet{lavi_motility_2020} which is studied in \citet{alazard_traveling_2022}.
	
	In the same way as in \citet{lavi_motility_2020}, we assume that polarity markers are rear markers, meaning that the rear part of the cell is defined by the area with the highest concentration of markers. These markers induce an active force at the edge of the cell given by $-\chi_c f_{\mathrm{act}}\left(c\right)$ with $\chi_c>0$ and $f_{\mathrm{act}}$ satisfying:
    \begin{subequations}\label{eq:conditions_f_act}
		\begin{align}
			& f_{\mathrm{act}} \in C^1\left(\mathbb{R^+}\right), \\
			& f_{\mathrm{act}} \text{ is an increasing function}, \\
			& f_{\mathrm{act}}\left(0\right) = 0, \\
			& \lim\limits_{x\rightarrow + \infty} f_{\mathrm{act}}\left(x\right) = L_c < + \infty.
		\end{align}
	\end{subequations}
	This force is used to model the action of the cytoskeleton on the cell. 
	
	The boundary condition \eqref{eq:model_pressure_boundary} is obtained as the balance of the forces acting on the membrane and can be seen as a perturbation of the Young-Laplace equation. $\gamma$ designates the effective surface tension and is given, $\kappa$ designates the curvature.
	
    To close the system, the internal solute transport problem is formulated in \eqref{eq:model_markers_dynamic} -- \eqref{eq:model_markers_boundary}. In the bulk $\Omega(t)$,  fast adsorption on the top and bottom plates (or onto an adhered cortex) is assumed. With rapid on and off rates, the quasi-2D transport dynamics are given by \eqref{eq:model_markers_dynamic} -- \eqref{eq:model_markers_boundary} where $a \in \left[0,1\right]$ is the steady fraction of adsorbed molecules not convected by the average flow and the effective diffusion coefficient is assumed to be $1$.

    In \eqref{eq:model_markers_boundary}, a zero solute flux on the moving boundary $\partial \Omega(t)$ is imposed. Simply put, the solute is effectively convected at a slower velocity than that of the fluid. Hence, its concentration decreases (increases) towards an advancing (retracting) front.

    The solute can be any cytoplasmic protein controlling the active force-generation or adhesion machinery. In this work, it is assumed that the concentration $c$ either induces an inwards pulling force or inhibits an outwards pushing force.
    
    Formally, the boundary condition \eqref{eq:model_markers_boundary} induces the conservation of molecular content over time, that is for all time $t$ it holds that:
    \begin{equation}\label{eq:M}
        M :=\int_{\Omega(t)} c(t,x,y)\dx x \dx y = \int_{\Omega_0} c_0(x,y)\dx x \dx y.
    \end{equation}
    
    From the incompressibility constraint \eqref{eq:model_incompressibility} it follows that the cell domain area $A_\Omega$ is constant over time. Let $\bu_{\mathrm{cm}}$ be the velocity of the center of mass. From the total force balance on the cell, we have that for all $t\geq 0$:
	\begin{equation}\label{eq:undercooling_velocity_cm}
		\bu_{\mathrm{cm}} \left(t\right) = - \dfrac{1}{A_{\Omega}} \int_{\partial \Omega\left(t\right)} \left(\chi_c f_{\mathrm{act}}\left(c\right) + \chi_u f_{\mathrm{und}} \left(V_n\right)\right) \bn \dx \sigma.
	\end{equation}

	Let us briefly comment the existing  literature concerning \eqref{eq:model}. Moving interface problems have raised many interesting and challenging mathematical issues. A well known example is the Stefan problem which describes the dynamics of the boundary between ice and water. In the biophysical community, we find a large number of free boundary models to describe tumor and tissue growth, cell motility and other phenomena. We refer to \citet{Aranson,Levine} for a review.
	
	Most of them are formulated through a fluid approach with surface tension. Some tumor growth models (e.g. \cite{Friedman_2004, Friedman_Hu_2006, Friedman_Reitich_2001}) resemble our model \eqref{eq:model}. However, there is an important difference: tumor growth naturally involves expanding domain while we consider here incompressible solutions. In the context of the motility of eukaryotic cells on substrates, various free boundary problems have been derived and studied, see \citet{Berlyand_2016, Berlyand_2017, Berlyand_2018, Berlyand_2019,Berlyand_2021,Mizuhara_2016,Mizuhara}. The existence of traveling wave solutions for these models is proved in \citet{Berlyand_2016,Berlyand_2018,Berlyand_2019, Berlyand_2021}.
	
	In the 1D setting, Keller-Segel system with free boundaries as a model for contraction driven motility were introduced and studied in \cite{Recho_2013-2,Recho_2013,Recho_2015,Recho_2018}.  
	
	In the context of sharp interface limit some models for cell motility were studied in \cite{CMM_2020,CMM_2022}. These models differ from our approach in that no polarity markers are involved. However, they are similar to our model in that they take into account an overcooling force at the edge of the cell (the sign is opposite to that of an undercooling force). We note that traveling wave solutions for a moving boundary problem of Hele-Shaw type have been studied in the case with kinetic undercooling regularization in \citet{Gunther}. This model, which models the movement of a bubble in an exterior Hele-Shaw flow, differs from our model in its point of view and in the absence of polarity markers and coupling with them.
	
	\subsection{On traveling waves}
	
	A remarkable feature of cell motility is the appearance of sustained movement in a given direction without external cue, see \cite{Barnhart, Aranson, Keren}.
    This phenomenon, known as spontaneous polarization, see \cite{calvez_analysis_2012,etchegaray_analysis_2017,PLOS} e.g., is mathematically described by the existence of traveling wave solutions and is the main subject of this article.

	A traveling wave is characterized by a domain with a fixed shape $\tilde{\Omega}$ moving at a constant velocity $V \in \mathbb{R}$ in a fixed direction $\bw \in \mathbb{R}^2$. We then have:
	\begin{equation*}
		\Omega\left(t\right) = \tilde{\Omega} + t V \bw.
	\end{equation*}
	Without loss of generality, assume that $V\geq0$ and $\bw=\left(1,0\right)$. In this case, the normal velocity at the boundary of the cell satisfies $V_n = V n_x$.
	
	Using the traveling wave ansatz:
	\begin{equation*}
		c = c\left(x-Vt, y\right), \quad \quad P = P\left(x-Vt, y\right), \quad \quad \Omega\left(t\right) = \tilde{\Omega} + \left(Vt,0\right),
	\end{equation*}
	we obtain that a solution of model \eqref{eq:model}, which is a traveling wave for this model, is defined as follows in \cref{def:tw}.
	\begin{defi} \label{def:tw}
		A traveling wave of the model \eqref{eq:model} is given by a domain $\tilde{\Omega} \subset \mathbb{R}^2$, $V\geq0$ a real velocity and two functions $P$ and $c$ satisfying:
		\begin{subequations}\label{eq:tw_def}
			\begin{align}[left=\empheqlbrace]
				&- \Delta P = 0 && \text{in } \tilde{\Omega}, \label{eq:tw_poisson}\\
				& P = \gamma \kappa + \chi_c f_{\mathrm{act}}\left(c\right) + \chi_u f_{\mathrm{und}}\left( V n_x \right) && \text{on } \partial\tilde{\Omega},  \label{eq:tw_boundary_pressure} \\
				& -\nabla P \cdot \bn  = V n_x && \text{on } \partial \tilde{\Omega},  \label{eq:tw_boundary_normal_velocity}\\
				& \opdiv \left( \left(V,0\right)c +  \nabla c + \left(1-a\right) \nabla P c\right) = 0 && \text{in }  \tilde{\Omega},  \label{eq:tw_concentration}  \\
				& \left( \nabla c + a \left(V,0\right) c \right)\cdot \bn = 0 && \text{on }  \partial \tilde{\Omega}. \label{eq:tw_concentration_boundary}
			\end{align}
		\end{subequations}
	\end{defi}
	
	\begin{proposition}\label{prop:characterization_tw}
		Let $V\geq0$ be given. If $c$ and $P$ are solutions of \cref{eq:tw_def} associated with $V$ then $c$ and $P$ are of the form:
		\begin{align*}
			& P\left(x,y\right) = p_1 - Vx && \text{with } p_1 \in \mathbb{R}, \\
			& c\left(x,y\right) = \dfrac{M}{\int_{\tilde{\Omega}} e^{-aVx'} \dx x' \dx y'} e^{-aVx}, &&\\
		\end{align*}
		where $\left(x,y\right) \in \tilde{\Omega}$ and $M \geq 0$ is the total quantity of markers.
		
		Furthermore, $\partial \tilde{\Omega}$ is characterized by the curvature equation, given for all $\left(x,y\right) \in \partial \tilde{\Omega}$:
		\begin{equation}\label{eq:tw_curvature_equation}
			\gamma \kappa\left(x,y\right) = p_1 - Vx - \chi_c f_{\mathrm{act}}\left( \frac{M e^{-aVx}}{\int_{\tilde{\Omega}} e^{-aVx'} \dx x' \dx y'} \right) - \chi_u f_{\mathrm{und}}\left(V n_x\right).
		\end{equation}
	\end{proposition}
	
	\begin{rema}
		We say that a traveling wave $\left(\tilde{\Omega},V\right)$ of \eqref{eq:model} is non-trivial  if the velocity is not zero. When it is non-trivial, we call it a traveling wave, while when it is trivial, we call it a resting state.
	\end{rema}
	
	By taking $V=0$ in \cref{prop:characterization_tw}, it follows that the model admits an unique resting state, which is radially symmetrical.
	\begin{proposition}\label{prop:characterization_es}
		The model \eqref{eq:model} admits an unique resting state given by:
		\begin{subequations}\label{eq:model_resting_state}
			\begin{align}[left=\empheqlbrace]
				& c^0\left(\bx\right) = \dfrac{M}{\left| \Omega \right|} && \bx \in \Omega^0, \label{eq:model_resting_state_concentration}\\
				& P^0\left(\bx\right) = \dfrac{\gamma}{R_0} + \chi_c f_{\mathrm{act}}\left(c^0\left(\bx\right)\right)&& \bx \in \Omega^0, \label{eq:model_resting_state_pressure}\\
				& \bu^0\left(\bx\right) = 0 && \bx \in \Omega^0, \label{eq:model_resting_state_velocity}\\
				& \Omega^0 = B\left(0,R_0\right),
			\end{align}
		\end{subequations}
		where $R_0 = \sqrt{\dfrac{A_{\Omega}}{\pi}}$. This resting state is radially symmetric. 
	\end{proposition}
	
	\subsection{Main results}
	
	First, we study the resting state given by \eqref{eq:model_resting_state} by performing a linear stability analysis around this resting state. The model depends on the parameters $a$, $\chi_c$, $\chi_u$, $M$, $\left|\Omega0\right|$ i.e. $R_0$ and the functions $f_{\mathrm{act}}$ and $f_{\mathrm{und}}$. We have chosen to consider $\chi_c$ as the bifurcation parameter, but we could also have considered $M$, $R_0$, $a$, $f'_{\mathrm{act}}\left(c^0\right)$ or $f'_{\mathrm{und}}\left(0\right)$. So in the following, to perform a stability analysis and obtain a bifurcation result, we should vary the value of the parameter $\chi_c$.
	We set:
	\begin{equation}\label{eq:def_chi_*}
		\chi_c^* = \dfrac{R_0 + \chi_u f'_{\mathrm{und}}\left(0\right)}{R_0 a c^0 f'_{\mathrm{act}}\left( c^0\right)}.
	\end{equation}
	
	We then have the following stability result.
	\begin{theo}[Stability of the resting state]\label{thm:undercooling_stability_resting_state}
			If $\chi_c < \chi_c^*$ then the resting state \eqref{eq:model_resting_state} is linearly stable. On the opposite, if $\chi_c > \chi_c^*$, then the resting state \eqref{eq:model_resting_state} is linearly unstable. 
	\end{theo}
	
	In order to validate the model's relevance to cell motility, we then investigate the existence of traveling waves. This is the aim of this paper. To do so, we propose an implicit approach based on a bifurcation argument. This approach allows us to clearly identify what happens to a cell with a fixed volume as $\chi_c$ increases.
	
	\begin{theo}[Existence of traveling waves]\label{thm:undercooling_implicit_tw}
		Assume that $f_{\mathrm{act}}$ satisfies assumptions \eqref{eq:conditions_f_act}. For all $a \in \left(0,1\right]$, $\gamma >0$, $R_0>0$ and $\chi_u>0$ there exists a one parameter family of traveling wave solutions $\left(\tilde{\Omega}_\chi, V_\chi\right)$ of \cref{eq:model}, parametrized by $\chi \in \left( \chi_c^*, + \infty\right)$ such that $\left|\tilde{\Omega}_\chi\right| = \pi R_0^2$.
	\end{theo}
	
	\subsection*{Plan of the paper}
	This work is organized as follows. First, in Section \ref{sec:char_tw} we prove the characterization of the traveling waves given by \cref{prop:characterization_tw} from which follows the characterization of the unique radially symmetric resting state (\cref{prop:characterization_es}). Then in Section \ref{sec:lin_stab}, we study the linear stability of the system and we prove \cref{thm:undercooling_stability_resting_state}. Section \ref{sec:tw_implicit} contain the proof of the existence of traveling waves by implicit construction (\cref{thm:undercooling_implicit_tw}). Finally, we give some conclusions.
    
    \section{Characterization of traveling waves} \label{sec:char_tw}
    
    In this section we prove the \cref{prop:characterization_tw} in order to characterize the traveling waves of the model. The proof is based on the \cref{def:tw} which defines traveling waves of the model. The \cref{prop:characterization_es} follows from the result by taking $V=0$.
    
    \begin{proof}[Proof of \cref{prop:characterization_tw}]
        We set $\tilde{\bu} = - \nabla P - \bu_{\mathrm{cm}}$ where $\bu_{\mathrm{cm}}$ is defined by \cref{eq:undercooling_velocity_cm}. The traveling wave assumption implies that $\bu_{\mathrm{cm}} = \left(V,0\right)$. Thus, setting $\phi = -P - Vx$ we have: 
		\begin{equation*}
			\tilde{\bu} = \nabla \phi.
		\end{equation*}
		Moreover, from \cref{eq:tw_poisson}, we deduce that $\phi$ satisfies $ - \Delta \phi = 0$ in $\tilde{\Omega}$ and from \eqref{eq:tw_boundary_normal_velocity} we deduce that $\phi$ also satisfies $\nabla \phi \cdot \bn = 0$ on $\partial \tilde{\Omega}$. On the one hand, we have that:
		\begin{equation*}
			\int_{\tilde{\Omega}}\left|\nabla \phi\right|^2 \dx \bx = - \int_{\tilde{\Omega}} \phi \Delta \phi \dx \bx + \int_{\partial \tilde{\Omega}} \phi \nabla \phi \cdot \bn \dx \sigma = 0.
		\end{equation*}
		On the another hand, we have that:
		\begin{equation*}
			\int_{\tilde{\Omega}}\left|\nabla \phi\right|^2 \dx \bx = \int_{\tilde{\Omega}} \left| \tilde{\bu}\right|^2  \dx \bx.
		\end{equation*}
		Thus we have that $\tilde{\bu} = 0$ on $\tilde{\Omega}$. It follows that for all $\left(x,y\right) \in \tilde{\Omega}$:
		\begin{equation*}
			P\left(x,y\right) = p_1 - Vx,
		\end{equation*}
		with $p_1 \in \mathbb{R}$.
		
		By substituting $\nabla P = - \left(V,0\right)$ in \cref{eq:tw_concentration,eq:tw_concentration_boundary}, the concentration of markers satisfies the following problem:
		\begin{subequations} \label{eq:tw_concentration_resolution}
			\begin{align}[left=\empheqlbrace]
				& \opdiv \left( \nabla c +a   \left(V,0\right)  c\right) = 0 && \text{in }  \tilde{\Omega},  \\
				& \left( \nabla c + a \left(V,0\right) c \right)\cdot \bn = 0 && \text{on }  \partial \tilde{\Omega}. 
			\end{align}
		\end{subequations}
		The non-negative solutions of \cref{eq:tw_concentration_resolution} are given by:
		\begin{equation}\label{eq:tw_concentration_solution}
			c\left(x,y\right) = c_1 e^{-aVx},
		\end{equation}
		with $c_1>0$ and $\left(x,y\right) \in \tilde{\Omega}$.
		Indeed, we can see that functions of the form \eqref{eq:tw_concentration_solution} are solutions of \cref{eq:tw_concentration_resolution}. If we assume that $c\left(x,y\right) = c_1\left(x,y\right) e^{-a Vx}$ is a solution of \cref{eq:tw_concentration_resolution}, then we have that:
		\begin{align*}
			0 & = \int_{\tilde{\Omega}} c_1\left(x,y\right) \opdiv\left(\nabla c\left(x,y\right) + a \left(V,0\right) c\left(x,y\right) \right) \dx x \dx y \\
			& = \int_{\tilde{\Omega}} \nabla c_1\left(x,y\right) \cdot \left(\nabla c\left(x,y\right) + a \left(V,0\right) c\left(x,y\right) \right) \dx x \dx y \\
			& = \int_{\tilde{\Omega}} \left| \nabla c_1\left(x,y\right) \right|^2 e^{-aVx} \dx x \dx y,
		\end{align*}
		from which we deduce that $\left| \nabla c_1 \right| = 0$ on $\tilde{\Omega}$.
		Finally, remember that the total number of markers $M$ is constant, so $c_1$ must be such that $\int_{\tilde{\Omega}} c_1 e^{-a Vx} \dx x \dx y = M$. This leads to 
		\begin{equation*}
			c_1 = \frac{M}{\int_{\tilde{\Omega}} c_1 e^{-a Vx'} \dx x' \dx y'}.
		\end{equation*}
		
		To find the curvature equation \eqref{eq:tw_curvature_equation}, we inject the expressions found for $P$ and $c$ into equation \eqref{eq:tw_boundary_pressure}.
    \end{proof}

	\section{Study of the resting state}\label{sec:lin_stab}
	In this section, we prove the \cref{thm:undercooling_stability_resting_state} about the stability of resting state. Inspired by \citet{alazard_traveling_2022}, we first linearize the problem \eqref{eq:model} around the resting state \eqref{eq:model_resting_state} (see \cref{lemme:model_undercooling_linearized_problem}). We derive an eigenvalue problem from this linearized problem and study its spectrum. To this end, we prove that, using Fourier analysis, we can decompose the study of the spectrum into the study of the spectrum of simpler problems (see \cref{lemme:model_undercooling_decomp_spectre}). Next, we demonstrate that when $\chi_c < \chi_c^*$ then the eigenvalues are all negative real parts (see \cref{lemme:model_undercooling_spectre_neg} and \cref{rem:model_undercooling_spectre_neg}). Finally, from the decomposition of the spectrum study, we derive an explicit condition on the eigenvalues (see \cref{lemme:model_undercooling_eigenvalue_cond}). Thus, when $\chi_c > \chi_c^*$, we can exhibit a positive real part eigenvalue. This proves the \cref{thm:undercooling_stability_resting_state}. 
	
	\begin{lemma} \label{lemme:model_undercooling_linearized_problem}
		The linearized problem associated to \cref{eq:model} around the resting state \eqref{eq:model_resting_state} is given by:
		\begin{subequations}\label{eq:model_linearized}
			\begin{align}[left=\empheqlbrace]
				& - \Delta \tilde{P} = 0 && \text{in } \Omega^0,  \\
				& \partial_t \rho\left(t,\theta\right) = - \partial_r \tilde{P}\left(t,R_0,\theta\right) && \theta \in \left(-\pi , \pi \right],  \\
				& \tilde{P} = - \frac{\gamma}{R_0^2} \left(\partial_{\theta \theta}^2 \rho + \rho\right) + \chi_c \tilde{c} f'_{\mathrm{act}}\left(c^0\right) - \chi_u \partial_r\tilde{P} f'_{\mathrm{und}}\left(0\right)  && \text{on } \partial \Omega^0,  \\
				& \partial_t \tilde{c} = \Delta \tilde{c} && \text{in } \Omega^0,  \\
				& \left( \nabla \tilde{c} - a  \nabla \tilde{P} \, c^0 \right)\cdot \mathbf{n} = 0 && \text{on } \partial\Omega^0.
			\end{align}
		\end{subequations}
	\end{lemma}
	
	\begin{proof}
		We perform a formal expansion of the solution $\left(c,\bu\right)$ near the resting state $\left(c^0, \bu^0\right)$. Let $\varepsilon >0$ small. For all $t\geq0$ we set:
		\begin{equation}\label{eq:model_omega_pertubed}
			\Omega\left(t\right) = \left\{\left(x,y\right) = \left(r \cos \theta, r\sin \theta\right) \text{ s.t. } 0 \leq r < R_0 + \varepsilon \rho\left(t,\theta\right) \text{ and } \theta \in \left(-\pi , \pi \right] \right\},
		\end{equation}
		and for all $\bx \in \Omega\left(t\right)$:
		\begin{align*}
			& c\left(t,\bx\right) = c^0 + \varepsilon \tilde{c} \left(t,\bx\right) + \mathcal{O}\left(\varepsilon^2\right), \\
			& P\left(t,\bx\right) = P^0 + \varepsilon \tilde{P} \left(t,\bx\right) + \mathcal{O}\left(\varepsilon^2\right),
		\end{align*}
		
		Using the fact that $P^0$ satisfies \cref{eq:model_resting_state_pressure} and $P$ satisfies \cref{eq:model_darcy,eq:model_incompressibility}, we deduce that:
		\begin{equation*}
			- \Delta \tilde{P} = 0, \quad \quad \text{ in } \Omega^0.
		\end{equation*}
		
		%From \cref{eq:model_omega_pertubed}, we can parameterise the boundary of the domain $\Omega$ and thus we obtain that, for all $\theta \in \left(-\pi,\pi\right]$ and $r \in \mathbb{R}_+$ such that $\left(r \cos \theta, r \sin \theta\right) \in \partial \Omega\left(t\right)$, the outward normal vector is expressed as:
		%\begin{align*}
		%	\bn\left(t,r,\theta\right) & = \dfrac{\nabla\left(r-\left(R_0 + \varepsilon \rho\left(t,\theta\right)\right)\right)}{\left| \nabla\left(r-\left(R_0 + \varepsilon \rho\left(t,\theta\right) \right)\right)\right|} \\
		%	& = \dfrac{1}{\left| \nabla\left(r-\left(R_0 + \varepsilon \rho\left(t,\theta\right) \right)\right)\right|} \begin{pmatrix}
		%		\cos \theta + \varepsilon \dfrac{\partial_\theta \rho\left(t,\theta\right)}{r} \sin \theta \\ \sin \theta - \varepsilon \dfrac{\partial_\theta \rho\left(t,\theta\right)}{r} \cos \theta
		%	\end{pmatrix},
		%\end{align*}
		%which leads to the fact for all $\theta \in \left(-\pi, \pi \right]$ we have:
		%\begin{equation*}
		%	\bn\left(t,\theta\right) = \begin{pmatrix}
		%		\cos \theta \\ \sin \theta
		%	\end{pmatrix} - \varepsilon \dfrac{\partial_\theta \rho \left(t,\theta\right)}{R_0} \begin{pmatrix}
		%		- \sin \theta \\  \cos \theta
		%	\end{pmatrix} + \mathcal{O} \left(\varepsilon ^2\right). 
		%\end{equation*}
		
		The equation \ref{eq:model_omega_pertubed} induces that the boundary of the domain $\Omega$ can be parametrized and thus we obtain that for all $\theta \in \left(-\pi, \pi \right]$ we have:
		\begin{equation*}
			\bn\left(t,\theta\right) = \begin{pmatrix}
				\cos \theta \\ \sin \theta
			\end{pmatrix} - \varepsilon \dfrac{\partial_\theta \rho \left(t,\theta\right)}{R_0} \begin{pmatrix}
				- \sin \theta \\  \cos \theta
			\end{pmatrix} + \mathcal{O} \left(\varepsilon ^2\right). 
		\end{equation*}

		Thus, on $\partial \Omega \left(t\right)$, we have:
		\begin{equation*}
			\nabla P \cdot \bn = \nabla P^0 \cdot \bn - \varepsilon \partial_r \tilde{P}  + \mathcal{O} \left(\varepsilon ^2\right)
		\end{equation*}
		and also:
		\begin{equation*}
			V_n = \varepsilon \partial_t \rho \left(t, \theta\right) + \mathcal{O} \left(\varepsilon ^2\right).
		\end{equation*}
		Using \cref{eq:model_normal_velocity}, we can therefore deduce that, for all $\theta \in \left(-\pi, \pi \right]$ and $t\geq 0$, we have:
		\begin{equation*}
			\partial_t \rho\left(t,\theta\right) = - \partial_r \tilde{P}\left(t,R_0,\theta\right)
		\end{equation*}
		
		The linearisation of the curvature is given for all $\theta \in \left(-\pi,\pi\right]$ and $t\geq 0$ by:
		\begin{equation*}
			\kappa\left(t,\theta\right) = \frac{1}{R_0} - \frac{\varepsilon}{R_0^2} \left(\partial_{\theta\theta}^2 \rho \left(t,\theta\right) + \rho \left(t,\theta\right)\right)  + \mathcal{O} \left(\varepsilon ^2\right).
		\end{equation*} 
		Since
		\begin{equation*}
			f_{\mathrm{act}}\left(c\right) = f_{\mathrm{act}}\left(c^0\right) + \varepsilon \tilde{c} f'_{\mathrm{act}}\left(c^0\right) + \mathcal{O} \left(\varepsilon ^2\right)
		\end{equation*}
		and
		\begin{equation*}
			f_{\mathrm{und}}\left(V_n\right) = \varepsilon \partial_t \rho\left(t,\theta\right) f'_{\mathrm{und}}\left(0\right) +  \mathcal{O} \left(\varepsilon ^2\right),
		\end{equation*}
		we have that \cref{eq:model_pressure_boundary} leads to that on $\partial \Omega^0$ we have:
		\begin{equation*}
			\tilde{P} = - \frac{\gamma}{R_0^2} \left(\partial_{\theta \theta}^2 \rho + \rho\right) + \chi_c \tilde{c} f'_{\mathrm{act}}\left(c^0\right) + \chi_u \partial_t \rho f'_{\mathrm{und}}\left(0\right).
		\end{equation*}
		
		Moreover, using the fact that $c^0$ satisfies \cref{eq:model_resting_state_concentration} and $c$ satisfies \cref{eq:model_markers_dynamic,eq:model_markers_boundary}, we deduce that in $\Omega^0$: 
		\begin{equation*}
			\partial_t \tilde{c} = \Delta \tilde{c}.
		\end{equation*}
		%We also have on $\partial \Omega^0$ that:
		%\begin{equation*}
		%	\left( \nabla c - a \nabla P \, c \right)\cdot \bn = \varepsilon \left(\nabla \tilde{c} - a \nabla \tilde{P} \, c^0\right) \cdot \bn + \mathcal{O}\left(\varepsilon^2\right),
		%\end{equation*}
		%which leads to:
		The boundary condition \eqref{eq:model_markers_boundary} leads to:
		\begin{equation*}
			\left(\nabla \tilde{c} - a \nabla \tilde{P} \, c^0\right) \cdot \bn = 0 \quad \quad \text{ on } \partial \Omega^0.
		\end{equation*}
	\end{proof}
	
	As the problem \eqref{eq:model_linearized} is radially symmetric, the spectral analysis can be performed using Fourier analysis.
	
	The eigenvalue problem associated with the linearized problem \eqref{eq:model_linearized} is given by:
	\begin{subequations}\label{eq:model_eigenvalue}
		\begin{align}[left=\empheqlbrace]
			& - \Delta \tilde{P} = 0 && \text{in } \Omega^0,  \\
			& \lambda \rho\left(\theta\right) = -\partial_r \tilde{P}\left(R_0,\theta\right) && \theta \in \left(-\pi , \pi \right],  \\
			& \tilde{P} = - \frac{\gamma}{R_0^2} \left(\partial_{\theta \theta}^2 \rho + \rho\right) + \chi_c \tilde{c} f'_{\mathrm{act}}\left(c^0\right) - \chi_u \partial_r\tilde{P} f'_{\mathrm{und}}\left(0\right)  && \text{on } \partial \Omega^0,  \\
			& \lambda \tilde{c} = \Delta \tilde{c} && \text{in } \Omega^0,  \\
			& \left( \nabla \tilde{c} - a  \nabla \tilde{P} \, c^0 \right)\cdot \mathbf{n} = 0 && \text{on } \partial\Omega^0,
		\end{align}
	\end{subequations}
	where $\lambda \in \mathbb{C}$.
	
	\begin{lemma}\label{lemme:model_undercooling_decomp_spectre} 
		Let $\lambda \in \mathbb{C}$. The eigenfunctions, satisfying \cref{eq:model_eigenvalue} associated with the eigenvalue $\lambda$ are of the form:
		\begin{align*}
			& \rho \left(\theta \right) = \sum_{m \in \mathbb{N}} \rho_{cm} \cos\left(m \theta\right) + \sum_{m \in \mathbb{N}} \rho_{sm} \sin\left(m \theta\right), \\
			& c \left(r,\theta \right) = \sum_{m \in \mathbb{N}} c_{cm}\left(r\right) \cos\left(m \theta\right) + \sum_{m \in \mathbb{N}} c_{sm}\left(r\right) \sin\left(m \theta\right), \\
			& P \left(r,\theta \right) = \sum_{m \in \mathbb{N}} P_{cm}\left(r\right) \cos\left(m \theta\right) + \sum_{m \in \mathbb{N}} P_{sm}\left(r\right) \sin\left(m \theta\right),
		\end{align*}
		with for all $m \in \mathbb{N}$, $\left(\rho_{cm}, c_{cm}, P_{cm}\right)$ (resp. $\left(\rho_{sm}, c_{sm}, P_{sm}\right)$) satisfying:
		\begin{subequations}\label{eq:model_eigenvalue_decomposition}
			\begin{align}[left=\empheqlbrace]
				& - \left(\partial_{rr}^2 + \frac{1}{r} \partial_r - \frac{m^2}{r^2} \right)  P_{cm} = 0 && r \in \left(0, R_0\right), \label{eq:model_eigenvalue_decomposition_lap_pressure} \\
				& \lambda \rho_{cm} = -\partial_r P_{cm}\left(R_0\right), && \label{eq:model_eigenvalue_decomposition_velocity_boundary} \\
				& P_{cm}\left(R_0\right)= \frac{\gamma}{R_0^2} \left(m^2 -1\right) \rho_{cm} + \chi_c c_{cm}\left(R_0\right) f'_{\mathrm{act}}\left(c^0\right) &&  \label{eq:model_eigenvalue_decomposition_pressure_boundary}\\
				& \quad \quad \quad \quad \quad \quad \quad \quad \quad \quad \quad \quad- \chi_u \partial_r P_{cm}\left(R_0\right) f'_{\mathrm{und}}\left(0\right), \nonumber &&  \\
				& \lambda c_{cm} = \left(\partial_{rr}^2 + \frac{1}{r} \partial_r - \frac{m^2}{r^2} \right) c_{cm} && r \in \left(0, R_0\right), \label{eq:model_eigenvalue_decomposition_concentration}  \\
				& \partial_r c_{cm}\left(R_0\right) - a  \partial_r P_{cm}\left(R_0\right) \, c^0  = 0. && \label{eq:model_eigenvalue_decomposition_concentration_boundary}
			\end{align}	
		\end{subequations}
		%	\begin{subequations}\label{eq:model_eigenvalue_decomposition}
			%		\begin{align}[left=\empheqlbrace]
				%			& - \Delta P_{m\lambda} = 0 && \text{in } \Omega^0, \label{eq:model_eigenvalue_decomposition_lap_pressure} \\
				%			& \lambda \rho_{m\lambda} \left(t,\theta\right) = -\partial_r P_{m\lambda}\left(t,R_0,\theta\right) && \theta \in \left(-\pi , \pi \right],  \\
				%			& P_{m\lambda}= \frac{\gamma}{R_0^2} \left(m^2 -1\right) \rho_{m\lambda} + \chi_c c_{m\lambda} f'_{\mathrm{act}}\left(c^0\right) - \chi_u \partial_r P_{m\lambda} f'_{\mathrm{und}}\left(0\right)  && \text{on } \partial \Omega^0,  \\
				%			& \lambda c_{m\lambda} = \Delta c_{m\lambda} && \text{in } \Omega^0,  \\
				%			& \left( \nabla c_{m\lambda} - a  \nabla P_{m\lambda} \, c^0 \right)\cdot \mathbf{n} = 0 && \text{on } \partial\Omega^0,
				%		\end{align}	
			%	\end{subequations}
		%	where $\rho_{m\lambda} \left(\theta\right)= \rho_{cm} \cos\left(m \theta\right)$, $c_{m\lambda} \left(r, \theta\right)= c_{cm}\left(r\right) \cos\left(m \theta\right)$ and $P_{m\lambda} \left(r, \theta\right)= P_{cm}\left(r\right) \cos\left(m \theta\right)$ (resp.  $\rho_{m\lambda} \left(\theta\right)= \rho_{sm} \sin\left(m \theta\right)$, $c_{m\lambda} \left(r, \theta\right)= c_{sm}\left(r\right) \sin\left(m \theta\right)$ and $P_{m\lambda} \left(r, \theta\right)= P_{sm}\left(r\right) \sin\left(m \theta\right)$).
	\end{lemma}
	
	\begin{proof}

		The result follows from the linearity of \cref{eq:model_eigenvalue}, the independence of the cosine and sine modes and the fact that for all $m \in \mathbb{N}$:
		\begin{equation*}
			\partial_{\theta \theta}^2 \rho_{m\lambda} = -m^2 \rho_{m\lambda}.
		\end{equation*}
	\end{proof}
	
	\begin{rema}
		We deduce from \cref{lemme:model_undercooling_decomp_spectre} and the independence of the cosine and sine perturbations that studying the eigenvalues of \cref{eq:model_eigenvalue} is equivalent to the study, for all $m \in \mathbb{N}$, of those of \cref{eq:model_eigenvalue_decomposition}. Indeed, if $\lambda \in \mathbb{C}$ is an eigenvalue of \cref{eq:model_eigenvalue_decomposition} associated with the eigenfunctions $\left(\rho_\lambda, c_\lambda, P_\lambda\right)$ then $\lambda$ is an eigenvalue of \cref{eq:model_eigenvalue} associated with the eigenfunctions:
		\begin{align*}
			& \rho \left(\theta\right)= \rho_{\lambda} \cos\left(m \theta\right), \\
			& c \left(r, \theta\right)= c_{\lambda}\left(r\right) \cos\left(m \theta\right), \\
			& P \left(r, \theta\right)= P_{\lambda}\left(r\right) \cos\left(m \theta\right).
		\end{align*}
	\end{rema}

	\begin{lemma}\label{lemme:model_undercooling_spectre_neg}
		If $m\geq 1$ and $0 \leq \chi_c \leq \frac{1}{a c^0 f'_{\mathrm{act}}\left(c^0\right)}$, then all the eigenvalues of \cref{eq:model_eigenvalue_decomposition} have non-positive real parts.
		%$0\leq\dfrac{\chi_c}{\chi_c^*}\leq1$
	\end{lemma}
	
	\begin{proof}
		Let $m \geq 1$. Let $\lambda \in \mathbb{C}$ be a eigenvalue of \cref{eq:model_eigenvalue_decomposition}. Let $\rho_m$, $c_m$ and $P_m$ satisfy \cref{eq:model_eigenvalue_decomposition}. For all $r \in \left(0,R_0\right)$ and $\theta \in \left(-\pi, \pi\right]$, we set:
		\begin{align*}
			& \rho \left(\theta\right)= \rho_{m} \cos\left(m \theta\right), \\
			& c \left(r, \theta\right)= c_{m}\left(r\right) \cos\left(m \theta\right), \\
			& P \left(r, \theta\right)= P_{m}\left(r\right) \cos\left(m \theta\right).
		\end{align*}
		We also set $Q = c - a c^0 P$. Thus $Q$ follows the following problem:
		\begin{subequations}\label{eq:model_eigenvalue_decomposition_Q}
			\begin{align}[left=\empheqlbrace]
				& -\Delta Q=\Delta c  &&\text{in }  \Omega^0,\\
				& \left(1-a \chi_c c^0 f'_{\mathrm{act}}\left(c^0\right)\right)P=\frac{\gamma}{R_0^2} \left( m^2-1\right) \rho +\chi_c f'_{\mathrm{act}}\left(c^0\right) Q + \lambda \chi_u f'_{\mathrm{und}}\left(0\right) \rho &&\text{on }  \partial \Omega^0 ,\\
				& \nabla Q \cdot \bn = 0 &&\text{on } \partial \Omega^0.
			\end{align}
		\end{subequations}
		
		Thus using \cref{eq:model_eigenvalue_decomposition,eq:model_eigenvalue_decomposition_Q} we compute:
		\begin{align*}
			\lambda \int_{\Omega^0} \left|c\right|^2 \dx \bx & = \int_{\Omega^0} \bar{c} \Delta c \dx \bx  =  \int_{\Omega^0} \left(\bar{Q} + a c^0 \bar{P}\right) \Delta Q \dx \bx \\
			& = - \int_{\Omega^0} \left| \nabla Q\right|^2 \dx \bx - a c^0\int_{\Omega^0} \nabla \bar{P} \cdot \nabla Q \dx \bx \\
			& = - \int_{\Omega^0} \left| \nabla Q\right|^2 \dx \bx - a c^0\int_{\partial \Omega^0} Q \nabla \bar{P} \cdot \bn \dx \sigma.
		\end{align*}
		Since 
		\begin{align*}
			\int_{\partial \Omega^0} Q \nabla \bar{P} \cdot \bn \dx \sigma & = \dfrac{1}{\chi_c f'_{\mathrm{act}}\left(c^0\right)} \left[ \int_{\partial \Omega^0} \left(1-a c^0 \chi_c f'_{\mathrm{act}}\left(c^0\right)\right)P \nabla \bar{P} \cdot \bn \dx \sigma \right. \\
			&\quad \quad \quad \quad \quad \quad \left. - \left(\frac{\gamma}{R_0^2}\left(m^2-1\right) + \lambda \chi_u f'_{\mathrm{und}}\left(0\right)\right)\int_{\partial \Omega^0} \rho \nabla \bar{P} \cdot \bn \dx \sigma \right] \\
			& = \dfrac{1}{\chi_c f'_{\mathrm{act}}\left(c^0\right)} \left[ \left(1-a c^0 \chi_c f'_{\mathrm{act}}\left(c^0\right) \right) \int_{\Omega^0} \left|\nabla P\right|^2 \dx \bx \right. \\
			& \quad \quad \quad \quad \quad \quad \quad \quad +\bar{\lambda} \frac{\gamma}{R_0^2}\left(m^2-1\right) \int_{\partial \Omega^0} \left|\rho\right|^2 \dx \sigma \\
			& \quad \quad \quad \quad \quad \quad \quad \quad \left. + \chi_u f'_{\mathrm{und}}\left(0\right)  \int_{\partial \Omega^0} \left|\nabla P \cdot \bn\right|^2 \dx \sigma \right].
		\end{align*}
		
		We deduce that for all $\lambda \in \mathbb{C}$ eigenvalue of $\mathcal{A}_m$ we have:
		\begin{multline*}
			\lambda \int_{\Omega^0} \left|c\right|^2 \dx \bx + \dfrac{\bar{\lambda}}{\chi_c f'_{\mathrm{act}} \left(\tilde{c}\right)} \frac{\gamma}{R_0^2} \left( m^2-1\right)\int_{\partial \Omega^0} \left|\rho\right|^2 \dx \sigma \\
			= - \int_{\Omega^0} \left| \nabla Q\right|^2 \dx \bx - a \tilde{c} \dfrac{1-a c^0 \chi_c f'_{\mathrm{act}}\left(c^0\right) }{\chi_c f'_{\mathrm{act}}\left(c^0\right)} \int_{\Omega^0} \left|\nabla P\right|^2 \dx \bx \\ - \dfrac{\chi_u f'_{\mathrm{und}}\left(0\right)}{\chi_c f'_{\mathrm{act}}\left(c^0\right)}  \int_{\partial \Omega^0} \left|\nabla P \cdot \bn\right|^2 \dx \sigma.
		\end{multline*}
		The result follows from this equality.
	\end{proof}
	
	\begin{rema}\label{rem:model_undercooling_spectre_neg}
		It seems that the result holds for all $\chi_c>0$ as long as $\chi_c < \chi_c^*$. We do not prove this here. However, the study of traveling waves in \cref{sec:tw_implicit} suggests that the result may indeed extend to all $\chi_c$ such that $\chi_c <\chi_c^*$.
	\end{rema}

	\begin{lemma}
		Let $m \in \mathbb{N}$. The eigenfunctions of \cref{eq:model_eigenvalue_decomposition} associated with the eigenvalue $\lambda \in \mathbb{C}$ are given by:
		\begin{equation*}
			\begin{pmatrix}
				\rho \\ c\left(r\right) \\ P\left(r\right)
			\end{pmatrix} = \begin{pmatrix}
				\hat{\rho}_{m \lambda} \\ \hat{c}_{m \lambda} I_m\left(-r \lambda^{\frac{1}{2}}\right) \\ \hat{P}_{m \lambda} r^m
			\end{pmatrix},
		\end{equation*}
		where $0\leq r < R_0$ and $\left(\hat{\rho}_{m \lambda}, \hat{c}_{m \lambda}, \hat{P}_{m \lambda} \right) \in \mathbb{C}^3$ solution of:
		\begin{subequations}\label{eq:model_eigenfunction}
			\begin{align}[left=\empheqlbrace]
				& \lambda \hat{\rho}_{m \lambda} = - m R_0^{m-1} \hat{P}_{m \lambda},\\
				%& \left[\lambda \left(1+ \frac{m}{R_0} \chi_u f'_{\mathrm{und}}\left(0\right)\right) + \frac{\gamma}{R_0^3} m \left(m^2-1\right) \right] \hat{\rho}_{m \lambda} = \frac{-m }{R_0} \chi_c f'_{\mathrm{act}}\left(c^0\right) I_m\left(-R_0 \lambda^{\frac{1}{2}}\right)\hat{c}_{m \lambda}, \\
				& R_0^m \hat{P}_{m \lambda} = \left(\frac{\gamma}{R_0^2} \left(m^2 - 1\right) + \lambda \chi_u f'_{\mathrm{und}}\left(0\right)\right) \hat{\rho}_{m \lambda} + \chi_c f'_{\mathrm{act}}\left(c^0\right) I_m\left(-R_0 \lambda^{\frac{1}{2}}\right) \hat{c}_{m \lambda}, \\
				& \frac{\lambda^{\frac{1}{2}}}{2} \left(I_{m-1}\left(-R_0 \lambda^{\frac{1}{2}}\right) + I_{m+1}\left(-R_0 \lambda^{\frac{1}{2}}\right) \right) \hat{c}_{m \lambda} = \lambda a c^0 \hat{\rho}_{m \lambda}.
			\end{align}
		\end{subequations}
	\end{lemma}
	
	\begin{proof}
		From \cref{lemme:model_undercooling_decomp_spectre}, we deduce that for all $0\leq r < R_0$ that:
		\begin{equation*}
			\rho = \hat{\rho}_{m \lambda} \in \mathbb{C}
		\end{equation*}
		and since $P$ satisfies \cref{eq:model_eigenvalue_decomposition_lap_pressure}, the Laplace equation:
		\begin{equation*}
			P\left(r \right) = \hat{P}_{m \lambda} r^m,
		\end{equation*}
		with $\hat{P}_{m \lambda} \in \mathbb{C}$.
		Moreover, as $c$ satisfies \cref{eq:model_eigenvalue_decomposition_concentration} and thus, by the definition of Bessel functions, we have:
		\begin{equation*}
			c\left(r \right) = \hat{c}_{m \lambda} I_m\left(-r \lambda^{\frac{1}{2}}\right),
		\end{equation*}
		where $0\leq r<R_0$, $\hat{c}_{m \lambda} \in \mathbb{C}$ and $I_m$ is the modified Bessel function of the first kind of order $m$.
		
		As $\rho$ and $P$ satisfy \cref{eq:model_eigenvalue_decomposition_velocity_boundary}, it follows that:
		\begin{equation*}
			\lambda \hat{\rho}_{m\lambda} = - m \hat{P}_{m \lambda} R_0^{m-1}.
		\end{equation*}
		Likewise, since $\rho$, $c$ and $P$ satisfy \cref{eq:model_eigenvalue_decomposition_pressure_boundary}, we have that:
		\begin{equation*}
			\hat{P}_{m\lambda} R_0^m = \frac{\gamma}{R_0^2} \left(m^2 - 1\right) \hat{\rho}_{m \lambda} + \chi_c f'_{\mathrm{act}}\left(c^0\right) \hat{c}_{m \lambda} I_m\left(-R_0 \lambda^{\frac{1}{2}}\right) - m \chi_u f'_{\mathrm{und}}\left(0\right) R_0^{m-1} \hat{P}_{m \lambda}.
		\end{equation*}
		Combining this with the previous equality, we obtain that:
		\begin{equation*}
			R_0^m \hat{P}_{m \lambda} = \left(\frac{\gamma}{R_0^2} \left(m^2 - 1\right) + \lambda \chi_u f'_{\mathrm{und}}\left(0\right)\right) \hat{\rho}_{m \lambda} + \chi_c f'_{\mathrm{act}}\left(c^0\right) I_m\left(-R_0 \lambda^{\frac{1}{2}}\right) \hat{c}_{m \lambda} .
		\end{equation*}
		Finally as $\rho$, $c$ and $P$ satisfy \cref{eq:model_eigenvalue_decomposition_concentration_boundary,eq:model_eigenvalue_decomposition_velocity_boundary}, we have:
		\begin{equation*}
			- \lambda^{\frac{1}{2}} I'_m\left(-R_0 \lambda^{\frac{1}{2}}\right)\hat{c}_{m \lambda} + \lambda a c^0 \hat{\rho}_{m \lambda} = 0.
		\end{equation*}
		Using the properties of Bessel functions, we deduce that:
		\begin{equation*}
			\frac{\lambda^{\frac{1}{2}}}{2} \left(I_{m-1}\left(-R_0 \lambda^{\frac{1}{2}}\right) + I_{m+1}\left(-R_0 \lambda^{\frac{1}{2}}\right)\right) \hat{c}_{m \lambda} =  \lambda a c^0 \hat{\rho}_{m \lambda}.
		\end{equation*}
	\end{proof}
	
	From the previous lemma, we derive the following result on the eigenvalues of \cref{eq:model_eigenvalue_decomposition}.
	
	\begin{lemma}\label{lemme:model_undercooling_eigenvalue_cond}
		Let $m\in \mathbb{N}$. Let $\lambda \in \mathbb{C}$ be an eigenvalue of \cref{eq:model_eigenvalue_decomposition}, then $\lambda$ is such that
		\begin{equation*}
			H_m\left(\lambda\right) = 0,
		\end{equation*}
		with $H_m$ defined for all $z \in \mathbb{C}$ by:
		\begin{multline*}%\label{eq:model_eigenvalue_cond}
			H_m\left(z\right) = z m \frac{a \chi_c c^0 f'_{\mathrm{act}}}{R_0} I_m\left(-R_0 z^{\frac{1}{2}}\right) \\
			+ \frac{z^{\frac{1}{2}}}{2} \left[z\left(1+\frac{m}{R_0}\chi_u f'_{\mathrm{und}}\left(0\right)\right) + \frac{\gamma}{R_0^3}m \left(m^2-1\right) \right] \left[I_{m-1}\left(-R_0 z^{\frac{1}{2}}\right)+I_{m+1}\left(-R_0 z^{\frac{1}{2}}\right)\right] .
		\end{multline*}
	\end{lemma}
	
	\begin{prop}
		\begin{enumerate}
			\item \label{item:undercooling_prop_stabilite} If $0 \leq \dfrac{\chi_c}{\chi_c^*}<1$, then \cref{eq:model_eigenvalue} admits $\lambda = 0$ as an eigenvalue of multiplicity three and all its other eigenvalues have a negative real part.
			\item \label{item:undercooling_prop_instabilite} If $\dfrac{\chi_c}{\chi_c^*}>1$, then \cref{eq:model_eigenvalue} admits a positive eigenvalue. 
		\end{enumerate}
	\end{prop}
	
	\begin{proof}
		When $m=0$, for all $z\in \mathbb{C}$, we have:
		\begin{equation*}
			H_0\left(z\right) = z^{\frac{3}{2}} I_1\left(-R_0 z^{\frac{1}{2}}\right).
		\end{equation*}
		$z=0$ is a solution of $H_0\left(z\right) = 0$. This eigenvalue is associated with the eigenfunctions defined for all $r\in \left[0, R_0\right)$ by:
		\begin{equation*}
			\begin{pmatrix}
				\hat{\rho}_{00}^1 \\ \hat{c}_{00}^1\left(r\right) \\ \hat{P}_{00}^1\left(r\right)
			\end{pmatrix} = \begin{pmatrix}
				1 \\ 0 \\ -\dfrac{\gamma}{R_0^2}
			\end{pmatrix} \quad \text{ and } \quad \begin{pmatrix}
				\hat{\rho}_{00}^2 \\ \hat{c}_{00}^2\left(r\right) \\ \hat{P}_{00}^2\left(r\right)
			\end{pmatrix} = \begin{pmatrix}
				0 \\ 1 \\ \chi_c f'_{\mathrm{act}} \left(c^0\right)
			\end{pmatrix}.
		\end{equation*} 
		
		The other roots of $H_0$ are given by $\lambda_{0k} = - \dfrac{x_{1k}^2}{R_0^2}$ where $x_{1k} \in \mathbb{R}$ is the $k^{\text{th}}$ root of $J_1$, the Bessel function of the first kind of order $1$. We have that the eigenvalue $\lambda_{0k}<0$ is associated with the eigenfunction defined for all $r\in \left[0, R_0\right)$ by:
		\begin{equation*}
			\begin{pmatrix}
				\hat{\rho}_{0k} \\ \hat{c}_{0k}\left(r\right) \\ \hat{P}_{0k}\left(r\right)
			\end{pmatrix} = \begin{pmatrix}
				0 \\ J_0\left(\frac{x_{1k}}{R_0} r  \right) \\ \chi_c f'_{\mathrm{act}} \left(c^0\right) J_0(x_{1k})
			\end{pmatrix}.
		\end{equation*} 
		
		When $m=1$, for all $z \in \mathbb{C}$ we have:
		\begin{multline*}
			H_1\left(z\right) = \frac{z^{\frac{3}{2}}}{2} \left(1+\frac{1}{R_0}\chi_u f'_{\mathrm{und}}\left(0\right)\right) \left[I_{0}\left(-R_0 z^{\frac{1}{2}}\right)+I_{2}\left(-R_0 z^{\frac{1}{2}}\right)\right] \\ + z  \frac{a \chi_c c^0 f'_{\mathrm{act}}}{R_0} I_1\left(-R_0 z^{\frac{1}{2}}\right)
		\end{multline*}
		Thus $z=0$ is a solution of $H_1\left(z\right) = 0$. This eigenvalue is associated with the eigenfunction defined for all $r \in \left[0,R_0\right)$ by:
		\begin{equation*}
			\begin{pmatrix}
				\hat{\rho}_{10} \\ \hat{c}_{10}\left(r\right) \\ \hat{P}_{10}\left(r\right)
			\end{pmatrix} = \begin{pmatrix}
				1 \\ 0 \\ 0
			\end{pmatrix}.
		\end{equation*} 
		
		If $m \geq 2$ then $z=0$ is solution of $H_m \left(z\right) = 0$. Nevertheless, since $\gamma >0$, from \cref{eq:model_eigenfunction} and $I_m\left(0\right) = 0$, we deduce that $\lambda =  0$ is associate to the zero eigenfunction defined for all $r \in \left[0,R_0\right)$ by:
		\begin{equation*}
			\begin{pmatrix}
				\hat{\rho}_{m0} \\ \hat{c}_{m0}\left(r\right) \\ \hat{P}_{m0}\left(r\right)
			\end{pmatrix} = \begin{pmatrix}
				0 \\ 0 \\ 0
			\end{pmatrix}
		\end{equation*} and $\lambda = 0$ is not an eigenvalue of \cref{eq:model_eigenvalue_decomposition} when $m\geq2$. 
		We can therefore conclude about \cref{item:undercooling_prop_stabilite} using \cref{rem:model_undercooling_spectre_neg}. 
		
		In order to exhibit a positive eigenvalue, we expand $H_1$ around $0$. For all $z$ close to $0$, we have:
		\begin{multline*}
			H_1\left(z\right)= \frac{z^{\frac{3}{2}}}{2} \left(1 + \frac{\chi_u f'_{\mathrm{und}}\left(0\right)}{R_0} - a c^0 \chi_c f'_{\mathrm{act}}\left(c^0\right) \right. \\ \left. + \frac{R_0^2 z}{8} \left(3+ \frac{3 \chi_u f'_{\mathrm{und}}\left(0\right)}{R_0} - a c^0 \chi_c f'_{\mathrm{act}}\left(c^0\right) \right)  \right) + \mathcal{O}\left(\left|z\right|^{\frac{7}{2}} \right)
		\end{multline*}
		The function $z \mapsto 1 + \frac{\chi_u f'_{\mathrm{und}}\left(0\right)}{R_0} - a c^0 \chi_c f'_{\mathrm{act}}\left(c^0\right) + \frac{R_0^2 z}{8} \left(3+ \frac{3 \chi_u f'_{\mathrm{und}}\left(0\right)}{R_0} - a c^0 \chi_c f'_{\mathrm{act}}\left(c^0\right) \right) $ admits $z_1$ as root, with $z_1$ defined by:
		\begin{equation*}
			z_1 = \dfrac{8\left(ac^0 \chi_c f'_{\mathrm{act}}\left(c^0\right) - 1 - \frac{\chi_u f'_{\mathrm{und}}\left(0\right)}{R_0} \right)}{R_0^2 \left(3 \left(1+\frac{\chi_u f'_{\mathrm{und}}\left(0\right)}{R_0}\right) - a c^0 \chi_c f'_{\mathrm{act}}\left(c^0\right)\right)}.
		\end{equation*}
		$z_1 \in \mathbb{R}$ changes sign and becomes positive when $\chi_c$ exceeds $\dfrac{R_0+\chi_u f'_{\mathrm{und}}\left(0\right)}{R_0 a c^0 f'_{\mathrm{act}}\left(c^0\right)} = \chi_c^*$. Furthermore, $z_1$ approaches a true root $\lambda_1$ of $H_1$. Indeed, when $a c^0 \chi_c f'_{\mathrm{act}}\left(c^0\right)$ close to $1 + \frac{\chi_u f'_{\mathrm{und}}\left(0\right)}{R_0}$, we have:
		\begin{equation*}
			\lambda_1 = \frac{4}{R_0^2} \left(a c^0 \chi_c f'_{\mathrm{act}}\left(c^0\right) - 1 + \frac{\chi_u f'_{\mathrm{und}}\left(0\right)}{R_0}\right) + o\left(\left|a c^0 \chi_c f'_{\mathrm{act}}\left(c^0\right) - 1 + \frac{\chi_u f'_{\mathrm{und}}\left(0\right)}{R_0}\right|\right)
		\end{equation*}
		associated with the eigenfunction defined for all $r \in \left[0,R_0\right)$ by:	
		\begin{equation*}
			\begin{pmatrix}
				\hat{\rho}_{1 \lambda_1} \\ \hat{c}_{1 \lambda_1}\left(r\right) \\ \hat{P}_{1 \lambda_1}\left(r\right)
			\end{pmatrix} = \renewcommand{\arraystretch}{3} \begin{pmatrix}
				\dfrac{-\chi_c f'_{\mathrm{act}}I_1\left(-R_0 \lambda_1^{\frac{1}{2}}\right)}{R_0 + \chi_u f'_{\mathrm{und}}\left(0\right)} \\ \lambda_1 I_1\left(-r \lambda_1^{\frac{1}{2}}\right) \\ \dfrac{\chi_c f'_{\mathrm{act}}I_1\left(-R_0 \lambda_1^{\frac{1}{2}}\right) \lambda_1}{R_0 + \chi_u f'_{\mathrm{und}}\left(0\right)} r
			\end{pmatrix}.
		\end{equation*} 
		Thus when $\chi_c >\chi_c^*$, \cref{eq:model_eigenvalue} admits $\lambda_1$ as positive eigenvalue.
	\end{proof}
	
	%\textcolor{red}{Ajouter conclusion/commentaire ?}

	\section{Existence of traveling waves} \label{sec:tw_implicit}
	
	In this section, we prove the \cref{thm:undercooling_implicit_tw} using a bifurcation argument. The arguments of the proof are inspired by those used in \citet{alazard_traveling_2022} in the case where $\chi_u = 0$. In order to prove the theorem, we define a functional $\mathcal{F}$ to which we can apply the Crandall-Rabinowitz bifurcation theorem \citep{crandall_bifurcation_1971} at the point $\chi_c = \chi_c^*$.
	
	Since the disc of radius $R_0$ with $V=0$ is a solution of \eqref{eq:tw_def}, we seek other solutions of \cref{eq:tw_def} such that the domain $\tilde{\Omega}$ is a perturbation of the disc. We therefore look for $\tilde{\Omega}$ such that:
	\begin{equation*}
		\tilde{\Omega} = \left\{ \left(r \cos\left(\theta\right), r \sin \left(\theta\right)\right) \text{ s.t. } 0 \leq r \leq R_0+ \rho\left(\theta\right) \text{ and } \theta \in \left(-\pi,\pi\right] \right\}
	\end{equation*}
	with $\rho : \mathbb{R} \rightarrow \left(-R_0, + \infty\right)$ a $2 \pi$-periodic function satisfying:
	\begin{equation*}
		\int_{-\pi}^{\pi} \left(R_0+\rho\left(\theta\right)\right)^2 - R_0^2 \dx \theta = 0.
	\end{equation*} This latter condition ensures that $\left|\tilde{\Omega}\right| = \pi R_0^2$.
	The boundary of domain $\tilde{\Omega}$ is parameterized by:
	\begin{equation*}
		\left\{ \left(R_0 + \rho\left(\theta\right) \right) \cos \theta , \left(R_0 + \rho\left(\theta\right) \right) \sin \theta \right\} \quad \text{ for } \theta \in \left(- \pi, \pi\right].
	\end{equation*}
	For all $\theta \in \left(-\pi, \pi \right]$, the outward normal vector to $\partial \tilde{\Omega}$ is given by:
	\begin{equation*}
		\bn\left(\theta\right) = \begin{pmatrix}
			n_1\left(\theta \right) \\ n_2\left(\theta\right)
		\end{pmatrix} = \dfrac{1}{\left(\left(R_0+\rho\left(\theta\right)\right)^2 + \rho'\left(\theta\right)^2\right)^{\frac{1}{2}}} \begin{pmatrix}
			\left(R_0+\rho\left(\theta\right)\right) \cos \theta + \rho'\left(\theta \right) \sin \theta \\ \left(R_0+\rho\left(\theta\right)\right) \sin \theta - \rho'\left(\theta \right) \cos \theta
		\end{pmatrix} 
	\end{equation*}
	and the mean curvature by:
	\begin{equation}\label{eq:undercooling_tw_mean_curvature}
		\kappa\left(\theta\right) = \dfrac{\left(R_0 + \rho\left(\theta\right)\right)^2 + 2 \rho'\left(\theta\right)^2 - \left(R_0 + \rho\left(\theta\right)\right) \rho''\left(\theta\right)}{\left(\left(R_0+\rho\left(\theta\right)\right)^2 + \rho'\left(\theta\right)^2\right)^{\frac{3}{2}}}.
	\end{equation}
	The condition on the boundary domain \cref{eq:tw_curvature_equation} can be expressed as:
	\begin{multline}\label{eq:undercooling_tw_curvature_equation_theta}
		\gamma \kappa \left(\theta\right) + V \left(R_0 + \rho\left(\theta\right)\right)\cos \theta + \chi_c f_{\mathrm{act}} \left(c_1\left(V,\rho\right) e^{-a V \left(R_0 + \rho\left(\theta\right)\right) \cos \theta}\right) \\ + \chi_u f_{\mathrm{und}}\left(V n_1\left(\theta\right) \right) = p_1,
	\end{multline}
	where $\theta \in \left(-\pi, \pi\right]$ and $c_1\left(V,\rho\right) = \dfrac{M}{\int_{-\pi}^{\pi} \int_{0}^{R_0 + \rho\left(\theta\right)} e^{-a V r \cos \theta} r \dx r \dx \theta}$.
	
	As we are looking for traveling waves in the $x$-direction, we restrict ourselves to domains $\tilde{\Omega}$ symmetrical about the $x$-axis. We then introduce the functional spaces:
	\begin{align*}
		& X = \left\{ \rho \in \mathcal{C}^{2,\alpha}_{\mathrm{per}}\left(-\pi,\pi\right) : \rho\left(\theta\right) = \rho\left(-\theta\right), \forall \theta \in \left(-\pi,\pi\right)  \right\}, \\
		& Y = \left\{ \rho \in \mathcal{C}^{0,\alpha}_{\mathrm{per}}\left(-\pi,\pi\right) : \rho\left(\theta\right) = \rho\left(-\theta\right), \forall \theta \in \left(-\pi,\pi\right)  \right\}.
	\end{align*}
	Therefore, the existence of a boundary $\partial \tilde{\Omega}$ solving \cref{eq:tw_curvature_equation} is equivalent to the existence of a function $\rho \in X$ solving \cref{eq:undercooling_tw_curvature_equation_theta}.
	
	Let $\mathcal{F}: \mathbb{R} \times X \times \mathbb{R} \times \mathbb{R} \rightarrow Y \times \mathbb{R} \times \mathbb{R} $ the functional defined for all $\left(\chi_c, \rho, V, p_1\right) \in \mathbb{R} \times X \times \mathbb{R} \times \mathbb{R}$ by:
	\begin{multline}\label{eq:undercooling_tw_bifurcation_F}
		\mathcal{F} \left(\chi_c, \rho, V, p_1\right) = \left( \gamma \kappa\left(\theta\right) + \chi_c f_{\mathrm{act}} \left(c_1\left(V,\rho\right) e^{-a V \left(R_0 + \rho\left(\theta\right)\right) \cos \theta}\right) \right. \\ + \chi_u f_{\mathrm{und}}\left(V n_1\left(\theta\right) \right) 
		+ V \left(R_0 + \rho\left(\theta\right)\right)\cos \theta - p_1 - \dfrac{\gamma}{R_0} \, ; \\
		\left. \int_{-\pi}^{\pi}  \left(R_0+\rho\left(\theta\right)\right)^2 - R_0^2 \dx \theta \, ; \int_{-\pi}^{\pi}  \rho\left(\theta\right) \cos \theta \dx \theta \right).
	\end{multline}
	
	%\textcolor{red}{Mettre seulement une définition de $\mathcal{F}$ avec une remarque expliquant ce que chaque terme signifie ?}
	
	\begin{lemma}
		The functional $\mathcal{F}$ defined by \eqref{eq:undercooling_tw_bifurcation_F} satisfies the following properties:
		\begin{enumerate}
			\item \label{item:undercooling_tw_F_1} $\mathcal{F}\left(\chi_c, 0, 0, 0\right) = 0$ for all $\chi_c \in \mathbb{R}$.
			\item \label{item:undercooling_tw_F_2} $\mathrm{Ker} \, \partial_{\left(\rho,V,p_1\right)} \mathcal{F}\left(\chi_c^*, 0,0,0\right)$ is a one dimensional subspace of $X \times \mathbb{R} \times \mathbb{R}$ spanned by $\left(0,1,0\right)$.
			\item \label{item:undercooling_tw_F_3} $\mathrm{Range} \, \partial_{\left(\rho,V,p_1\right)}  \mathcal{F}\left(\chi_c^*, 0,0,0\right)$ is a closed subspace of $Y \times \mathbb{R} \times \mathbb{R}$ of codimension $1$.
			\item \label{item:undercooling_tw_F_4} $\partial_{\chi_c} \partial_{\left(\rho,V,p_1\right)}  \mathcal{F}\left(\chi_c^*, 0,0,0\right)\left[ 0, 1, 0 \right] \notin \mathrm{Range} \, \partial_{\left(\rho,V,p_1\right)}  \mathcal{F}\left(\chi_c^*, 0,0,0\right)$.
		\end{enumerate}
	\end{lemma}
	
	\begin{proof}
		\textit{Item \ref{item:undercooling_tw_F_1}.} When $\rho = 0$, using the expression of the mean curvature \eqref{eq:undercooling_tw_mean_curvature}, we have for all $\theta \in \left(-\pi,\pi\right]$, $\kappa\left(\theta\right) = \frac{1}{R_0}$. Thus for all $\chi_c \in \mathbb{R}$, we have $\mathcal{F}\left(\chi_c, 0, 0, 0\right) = 0$.
		
		\textit{Item \ref{item:undercooling_tw_F_2}.} 
		Let $\chi_c \in \mathbb{R}$. Let $\mathcal{L}_{\chi_c}$ be the linear operator defined by:
		\begin{equation*}
			\begin{array}{ccrcl}
				\mathcal{L}_{\chi_c} & : & X \times \mathbb{R} \times \mathbb{R} & \to & Y \times \mathbb{R} \times \mathbb{R} \\
				& & \left(\rho,V,p_1\right) & \mapsto & \mathcal{F}_\rho \left(\chi_c,0,0,0\right)\left[\rho \right] + \mathcal{F}_V \left(\chi_c,0,0,0\right)\left[V \right] + \mathcal{F}_{p_1} \left(\chi_c,0,0,0\right)\left[p_1 \right].
			\end{array}
		\end{equation*}
		For all $\left(\rho, V, p_1\right) \in X \times \mathbb{R} \times \mathbb{R}$ we have:
		\begin{multline*}
			\mathcal{L}_{\chi_c} \left(\rho,V,p_1\right) = \left( -\gamma \dfrac{\rho + \rho''}{R_0^2} - \dfrac{\chi_c c^0 f'_{\mathrm{act}}\left(c^0\right)}{R_0} \int_{-\pi}^{\pi} \rho\left(\theta\right) \dx \theta - a c^0 \chi_c f'_{\mathrm{act}}\left(c^0\right) R_0 V \cos \theta \right.\\
			+ R_0 V \cos \theta - p_1 + \chi_u f'_\mathrm{und}\left(0\right) V \cos \theta \, ; \\
			\left. 2 R_0 \int_{-\pi}^{\pi} \rho \left(\theta\right) \dx \theta \, ;  \int_{-\pi}^{\pi}  \rho\left(\theta\right) \cos \theta \dx \theta \right)
		\end{multline*}
		and thus in particular for $\chi_c = \chi_c^*$ we have:
		\begin{multline*}
			\mathcal{L}_{\chi_c^*} \left(\rho,V,p_1\right) = \left( -\gamma \dfrac{\rho + \rho''}{R_0^2} - \dfrac{R_0 + \chi_u f'_{\mathrm{und}}\left(0\right)}{a R_0^2} \int_{-\pi}^{\pi} \rho\left(\theta\right) \dx \theta - p_1 \, ;\right.\\
			\left. 2 R_0 \int_{-\pi}^{\pi} \rho \left(\theta\right) \dx \theta \, ;  \int_{-\pi}^{\pi}  \rho\left(\theta\right) \cos \theta \dx \theta \right).
		\end{multline*}
		The elements $\left(\rho,V,p_1\right)$ of $\mathrm{Ker} \, \partial_{\left(\rho,V,p_1\right)} \mathcal{F}\left(\chi_c^*, 0,0,0\right)$ are therefore such that
		\begin{equation}\label{eq:undercooling_tw_edo_ker}
			\rho'' + \rho = -\frac{R_0^2}{\gamma} p_1,
		\end{equation}
		with the conditions
		\begin{align*}[left=\empheqlbrace]
			& \int_{-\pi}^{\pi} \rho\left(\theta\right) \dx \theta = 0, \\
			& \int_{-\pi}^{\pi} \rho\left(\theta\right) \cos \theta \dx \theta = 0, \\
			& \rho \text{ even function}.
		\end{align*}
		As $ \rho$ even, necessarily the condition $\int_{-\pi}^{\pi} \rho\left(\theta\right) \sin \theta \dx \theta = 0$ holds.
		The solutions to the differential equation \eqref{eq:undercooling_tw_edo_ker} are given, for all $\theta \in \left(-\pi,\pi\right]$, by 
		\begin{equation*}
			\rho\left(\theta\right) = A \cos \theta + B \sin \theta - \frac{R_0^2}{\gamma} p_1
		\end{equation*}where $A, B \in \mathbb{R}$. 
		
		We have:
		\begin{equation*}
			\int_{-\pi}^{\pi} \rho\left(\theta\right) \dx \theta = - 2 \pi \frac{R_0^2}{\gamma} p_1
		\end{equation*}
		and thus necessarily $p_1 = 0$. This leads to:
		\begin{align*}
			& \int_{-\pi}^{\pi} \rho\left(\theta\right) \cos \theta \dx \theta = A \pi, \\
			& \int_{-\pi}^{\pi} \rho\left(\theta\right) \sin \theta \dx \theta = B \pi
		\end{align*} which implies $A= 0$ and $B=0$. We then have $\rho = 0$ and  
		\begin{equation*}
			\mathrm{Ker} \, \partial_{\left(\rho,V,p_1\right)} \mathcal{F}\left(\chi_c^*, 0,0,0\right) = \mathrm{Span}\left\{ \left(0,1,0\right)\right\}.
		\end{equation*}
		
		%\textcolor{red}{détails à ajouter pour le calcul de $\mathcal{L}_{\chi_c}$}
		
		\textit{Item \ref{item:undercooling_tw_F_3}.} 
		To prove this property, we demonstrate that:
		\begin{equation*}
			\mathrm{Range} \, \partial_{\left(\rho, V, p_1\right)} \mathcal{F} \left(\chi_c^*, 0, 0, 0\right) = \left\{ \left(h, C_1,C_2\right)\in Y \times \mathbb{R} \times \mathbb{R} \text{ s.t. } \int_{-\pi}^{\pi} h\left(\theta\right) \cos \theta \dx \theta =0\right\}.
		\end{equation*}	
		Let $\left(h, C_1,C_2\right) \in \mathrm{Range} \, \partial_{\left(\rho, V, p_1\right)} \mathcal{F} \left(\chi_c^*, 0, 0, 0\right)$. Thus there exists $\rho \in X$ and $p_1 \in \mathbb{R}$ such that:
		\begin{equation}\label{eq:undercooling_range_demo}
			-\gamma \dfrac{\rho + \rho''}{R_0^2} - \dfrac{R_0 + \chi_u f'_{\mathrm{und}}\left(0\right)}{a R_0^2} \int_{-\pi}^{\pi} \rho\left(\theta\right) \dx \theta - p_1 = h.
		\end{equation}
		By multiplying it by $\cos \theta$ and integrating it over $\left[- \pi, \pi\right]$ we obtain that necessarily $h$ is such that:
		\begin{equation*}
			\int_{-\pi}^{\pi} h\left(\theta\right) \cos \theta \dx \theta =0.
		\end{equation*}
		
		Reciprocally, let $h \in \mathcal{C}_{\mathrm{per}}^{0,\alpha}\left(-\pi,\pi\right)$ and $p_1 \in \mathbb{R}$. The \cref{eq:undercooling_range_demo} admits a soution in $\mathcal{C}_{\mathrm{per}}^{2,\alpha}\left(-\pi,\pi\right)$ if and only if $\int_{-\pi}^{\pi} h\left(\theta\right) \cos \theta \dx \theta = \int_{-\pi}^{\pi} h\left(\theta\right) \sin \theta \dx \theta =0$. The general solutions are of the form:
		\begin{equation*}
			\rho\left(\theta\right) = \overline{\rho}\left(\theta\right) + k_1 \cos \left(\theta\right) + k_2 \sin \left(\theta\right),
		\end{equation*}
		where $\overline{\rho}$ is an even particular solution.
		
		Let $\left(h, C_1, C_2\right) \in \mathrm{Range} \, \partial_{\left(\rho, V, p_1\right)} \mathcal{F} \left(\chi_c^*, 0, 0, 0\right)$. As $h \in Y$ and is even, necessarily we have that $\int_{- \pi}^\pi h\left(\theta\right) \sin \theta \dx \theta =0$. Assume that $\int_{-\pi}^{\pi} h\left(\theta\right) \cos \theta \dx \theta =0$. There exists $\rho \in X$ solution of  \cref{eq:undercooling_range_demo}. Taking the even part of the general solutions, we obtain that:
		\begin{equation*}
			\rho\left(\theta \right) = \dfrac{1}{2} \left[\overline{\rho}\left(\theta\right) + \overline{\rho}\left(-\theta\right)\right] + k_1 \cos \left(\theta\right).
		\end{equation*}
		We choose $k_1$ such that $\int_{- \pi}^\pi \rho \left(\theta\right) \cos \theta \dx \theta = C_2$.
		
		Integrating \cref{eq:undercooling_range_demo} with respect to $\theta$ yields to:
		\begin{equation*}
			\left(-\dfrac{\gamma}{R_0^2} - \dfrac{2 \pi \left(R_0 + \chi_u f'_{\mathrm{und}}\left(0\right)\right)}{a R_0^2} \right) \int_{- \pi}^{\pi} \rho\left(\theta \right) \dx \theta - 2 \pi p_1 = \int_{-\pi}^{\pi} h\left(\theta\right) \dx \theta.
		\end{equation*}
		We thus choose $p_1$ such that $\int_{-\pi}^{\pi} \rho\left(\theta\right) \dx \theta = C_1$.

		\textit{Item \ref{item:undercooling_tw_F_4}.}
		For all $\left(\chi_c, \rho, V, p_1\right) \in \mathbb{R} \times X \times \mathbb{R} \times \mathbb{R}$ we have:
		\begin{equation*}
			\partial_{\chi_c} \mathcal{L}_{\chi_c} \left(\rho,V,p_1\right) = \left(-\frac{c^0 f'_{\mathrm{act}}\left(c^0\right)}{R_0} \int_{-\pi}^{\pi} \rho\left(\theta\right) \dx \theta -a c^0 f'_{\mathrm{act}}\left(c^0\right) R_0 V \cos \theta \, ; 0 \, ; 0 \, ; 0\right).
		\end{equation*}
		In particular, we have for $\chi_c = \chi_c^*$ and $\left(\rho, V, p_1\right) = \left(0,1,0 \right)$:
		\begin{equation*}
			\partial_{\chi_c} \mathcal{L}_{\chi_c^*} \left(0,1,0\right) = \left( -a c^0 f'_{\mathrm{act}}\left(c^0\right) R_0 \cos \theta\, ; 0 \, ; 0 \, ; 0\right).
		\end{equation*}
		By contradiction, assume that 
		\begin{equation*}
		    \partial_{\chi_c} \partial_{\left(\rho,V,p_1\right)}  \mathcal{F}\left(\chi_c^*, 0,0,0\right)\left[ 0, 1, 0 \right] \in \mathrm{Range} \, \partial_{\left(\rho,V,p_1\right)}  \mathcal{F}\left(\chi_c^*, 0,0,0\right).
		\end{equation*}
		Then there exists $\left(\rho, V, p_1 \right) \in X \times \mathbb{R} \times \mathbb{R}$ such that:
		\begin{align*}[left=\empheqlbrace]
			& \gamma \dfrac{\rho + \rho''}{R_0^2} + \dfrac{R_0 + \chi_u f'_{\mathrm{und}}\left(0\right)}{a R_0^2} \int_{-\pi}^{\pi} \rho\left(\theta\right) \dx \theta + p_1 = a c^0 f'_{\mathrm{act}}\left(c^0\right) R_0 \cos \theta \\
			& \int_{-\pi}^{\pi} \rho \left(\theta\right) \dx \theta =  \int_{-\pi}^{\pi}  \rho\left(\theta\right) \cos \theta \dx \theta = \int_{-\pi}^{\pi}  \rho\left(\theta\right) \sin \theta \dx \theta = 0.
		\end{align*}
		By multiplying by $\cos\left(\theta\right)$ and integrating the first line over $\left[-\pi, \pi\right]$, we obtain:
		\begin{multline*}
			\gamma \int_{-\pi}^{\pi} \rho\left(\theta\right) \cos\theta \dx \theta + \gamma \int_{-\pi}^{\pi} \rho''\left(\theta\right) \cos\theta \dx \theta \\ + \left(\dfrac{R_0 + \chi_u f'_{\mathrm{und}}\left(0\right)}{a} \int_{-\pi}^{\pi} \rho\left(\theta\right) \dx \theta + p_1 \right) \int_{-\pi}^{\pi} \cos\theta \dx \theta = a c^0 f'_{\mathrm{act}}\left(c^0\right) R_0^3 \int_{-\pi}^{\pi} \cos^2 \theta \dx \theta.
		\end{multline*}
		Since $\rho \in X$, $\rho$ is $2\pi$-periodic and thus $\int_{-\pi}^{\pi} \rho''\left(\theta\right) \cos\theta \dx \theta = \int_{-\pi}^{\pi} \rho\left(\theta\right) \cos\theta \dx \theta$. Then we have:
		\begin{equation*}
			a c^0 f'_{\mathrm{act}}\left(c^0\right) R_0^3 \pi = 0,
		\end{equation*} which is a contradiction as $a>0$, $f'_{\mathrm{act}} \left(c^0\right) >0$, $c^0>0$ and $R_0>0$.
	\end{proof}
	
	We can therefore apply the Crandall-Rabinowitz bifurcation theorem \citep{crandall_bifurcation_1971} to the functional $\mathcal{F}$ around the bifurcation point $\left(\chi_c^*, 0, 0, 0\right)$. Then for any complement $Z=Z_1 \times Z_2 \times Z_3$ of $\mathrm{Ker} \, \mathcal{F}\left(\chi_c^*, 0, 0, 0\right)$ in $X \times \mathbb{R} \times \mathbb{R}$, there exists a neighborhood $N$ of $\left(\chi_c^*, 0, 0, 0\right)$ in $\mathbb{R} \times X \times \mathbb{R} \times \mathbb{R}$, an interval $I = \left(-\varepsilon, \varepsilon\right)$ for some $\varepsilon>0$ and four continuous functions $\varphi : I \rightarrow \mathbb{R}$, $\psi_1 : I \rightarrow Z_1$, $\psi_2 : I \rightarrow Z_2$ and $\psi_3 : I \rightarrow Z_3$ such that for all $s \in I$:
	\begin{equation*}
		\mathcal{F}\left(\varphi\left(s\right), \psi_1\left(s\right), \psi_2\left(s\right), \psi_3\left(s\right)\right) = \left(0,0,0,0\right)
	\end{equation*}
	and \begin{equation*}
		\varphi\left(0\right) = \chi_c^*, \quad \psi_1\left(0\right) = 0, \quad \psi_2\left(0\right) = 0, \quad \psi_3\left(0\right) = 0.
	\end{equation*}
	We can note that for all $s \in \left(-\varepsilon, \varepsilon\right)$, $\psi_1\left(s\right)$ denotes a function.
	
	\noindent In particular, the Crandall-Rabinowitz bifurcation theorem implies that the solutions $\left(\chi_c, \rho, V, p_1\right)$ of the equation $\mathcal{F} \left(\chi_c, \rho, V, p_1\right) = \left(0,0,0,0\right)$ are of the form:
	%\begin{subequations}\label{eq:undercooling_tw_bifurcation_solutions}
	\begin{align}[left= \empheqlbrace]
		& \chi_c\left(s\right) = \varphi\left(s\right), \\
		& \rho \left(s, \theta \right) = s \psi_1 \left(s\right)\left(\theta\right), \\
		& V\left(s\right) = s + s \psi_2\left(s\right), \\
		& p_1\left(s\right) = s \psi_3\left(s\right),
	\end{align}
	%\end{subequations}
	where $s \in I$ and $\theta \in \left(-\pi,\pi\right]$. In addition, they satisfy for all $\theta \in \left(-\pi,\pi\right]$:
	\begin{align*}
		& \chi_c\left(0\right) = \chi_c^*, \\
		& \rho \left(0, \theta \right) = 0 \quad \text{and} \quad \partial_s \rho \left(0, \theta \right) = 0,  \\
		& V\left(0\right) = 0 \quad \text{and} \quad V'\left(0\right) = 1, \\
		& p_1\left(0\right) = 0 \quad \text{and} \quad p_1'\left(0\right) = 0.
	\end{align*}
	
	Thus we have for all $s \in I$:
	\begin{equation*}
		V\left(s\right) = s + o \left(s\right).
	\end{equation*}
	
	As $\left(\chi_c, \rho, V, p_1\right)$ is solution of the equation $\mathcal{F} \left(\chi_c, \rho, V, p_1\right) = \left(0,0,0,0\right)$, we have for all $s \in I$:
	\begin{equation*}
		\int_{-\pi}^\pi \rho\left(s,\theta\right) \cos \theta \dx \theta = 0 \quad \text{and} \quad \int_{-\pi}^\pi \left(R_0 + \rho\left(s,\theta\right) \right)^2 - R_0^2\dx \theta = 0.
	\end{equation*}
	Thus we deduce that for all $n \in \mathbb{N}$ we have:
	\begin{equation*}
		\int_{-\pi}^{\pi} \partial_s^n \rho\left(0,\theta\right) \cos \theta \dx \theta = 0,
	\end{equation*}
	and 
	\begin{equation*}
		\int_{- \pi}^{\pi} \partial_{s}^2 \rho\left(0,\theta \right) \dx \theta = 0.
	\end{equation*}
	
	\begin{lemma}
		Assume
		\begin{equation*}
			\partial_\theta^n \rho\left(0,\theta\right) = \partial_\theta^n \partial_s \rho\left(0,\theta\right) = 0, \quad \forall n \in \mathbb{N}, \forall \theta \in \left(-\pi,\pi\right].
		\end{equation*}
		Then, we have $\chi_c'\left(0\right) = 0$.
	\end{lemma}
	
	\begin{proof}
		For all $s\in I$ and $\theta \in \left(-\pi,\pi\right]$, we set 
		\begin{equation*}
			z\left(s,\theta\right) = c_1\left(V\left(s\right),\rho\left(s,\theta\right)\right) e^{-a V\left(s\right) \left(R_0 + \rho\left(s,\theta\right)\right) \cos \theta}.
		\end{equation*}
		We thus have:
		\begin{equation*}
			z\left(0,\theta\right) = c_1\left(V\left(0\right),\rho\left(0,\theta\right)\right) e^{-a V\left(0\right) \left(R_0 + \rho\left(0,\theta\right)\right) \cos \theta} = c^0.
		\end{equation*}
		We differenciate with respect to $s$ the first component of $\mathcal{F}$ and we have for all $s\in I$ and $\theta \in \left(-\pi,\pi\right]$:
		\begin{multline}\label{eq:undercooling_bifurcation_ds_F1}
			0 = \gamma \partial_s \kappa\left(s,\theta\right) + \chi_c'\left(s\right) f_{\mathrm{act}}\left(z\left(s,\theta\right)\right) + \chi_c\left(s\right) f'_{\mathrm{act}}\left(z\left(s,\theta\right)\right) \partial_s z\left(s,\theta\right) \\ + \chi_u f'_{\mathrm{und}}\left(V\left(s\right) n_1\left(s,\theta\right)\right)\left(V'\left(s\right) n_1\left(s,\theta\right) + V\left(s\right) \partial_s n_1\left(s,\theta\right)\right) \\ + V'\left(s\right)\left(R_0+\rho\left(s,\theta\right)\right) \cos \theta + V\left(s\right) \partial_s \rho\left(s,\theta\right) \cos \theta - p'_1\left(s\right).
		\end{multline}
		Evaluating it at $s=0$, it leads to:
		\begin{equation*}
			0 = \gamma \partial_s \kappa\left(0,\theta\right) + \chi_c'\left(0\right) f_{\mathrm{act}}\left(c^0\right) + \chi_c^* f'_{\mathrm{act}}\left(c^0\right) \partial_s z\left(0,\theta\right) + \chi_u f'_{\mathrm{und}}\left(0\right)n_1\left(0,\theta\right) + R_0\cos \theta.
		\end{equation*}
		
		As $\partial_\theta \rho\left(s,\theta\right) = s \partial_\theta \psi_1\left(s,\theta\right)$, we have $\partial_\theta \rho\left(0,\theta\right) = 0$. We then have:
		\begin{equation*}
			n_1\left(0,\theta\right) = \dfrac{1}{R_0}\left(R_0 \cos \theta + \partial_\theta \rho\left(0,\theta\right) \sin \theta\right) = \cos \theta.
		\end{equation*}
		Moreover, as for all $s\in I$ and $\theta \in \left(-\pi,\pi\right]$ we have:
		\begin{multline*}
			\partial_s z\left(s,\theta\right) = \left( \partial_1 c_1\left(V\left(s\right), \rho\left(s,\theta\right) \right) V'\left(s\right) + \partial_2 c_1\left(V\left(s\right),\rho\left(s,\theta\right)\right) \partial_s \rho\left(s,\theta\right) \right) e^{-a V\left(s\right) \left(R_0 + \rho\left(s,\theta\right)\right) \cos \theta} \\ -a \left(V'\left(s\right)\left(R_0 + \rho\left(s,\theta\right)\right) + V\left(s\right)\partial_s \rho\left(s,\theta\right)\right) \cos \theta c_1\left(V\left(s\right), \rho\left(s,\theta\right)\right) e^{-a V\left(s\right) \left(R_0 + \rho\left(s,\theta\right)\right) \cos \theta}.
		\end{multline*}
		As $\partial_1 c_1\left(0,0\right) = 0$, it follows that:
		\begin{equation*}
			\partial_s z\left(0,\theta\right) = - a R_0 \cos \theta c^0.
		\end{equation*}
		Finally, using the definition of $\chi_c^*$ (cf \cref{eq:def_chi_*}), we have for all $\theta \in \left(-\pi,\pi\right]$:
		\begin{align*}
			0 & = \gamma \partial_s \kappa\left(0,\theta\right) + \chi'_c\left(0\right) f_{\mathrm{act}}\left(c^0\right) + \left(\chi_u f'_{\mathrm{und}}\left(0\right) + R_0 - \chi_c^* a R_0 c^0 f'_{\mathrm{act}}\left(c^0\right)\right) \cos \theta \\ & = \gamma \partial_s \kappa\left(0,\theta\right) + \chi'_c\left(0\right) f_{\mathrm{act}}\left(c^0\right).
		\end{align*}
		As for all $\theta \in \left(-\pi,\pi\right]$ we have $\partial_s \kappa\left(0,\theta\right) = 0$, it follows that necessarily $\chi'_c\left(0\right) = 0$.
	\end{proof}

	\begin{lemma}
		Assume
		\begin{equation*}
			\partial_\theta^n \rho\left(0,\theta\right) = \partial_\theta^n \partial_s \rho\left(0,\theta\right) = \partial_\theta^n \partial_s^2 \rho\left(0,\theta\right) = 0, \quad \forall n \in \mathbb{N}, \forall \theta \in \left(-\pi,\pi\right].
		\end{equation*}
		Then, we have 
		\begin{equation*}
			\chi''_c\left(0\right) = -\dfrac{\left(R_0 + \chi_u f'_{\mathrm{und}}\left(0\right)\right) a R_0}{2 \left( f'_{\mathrm{act}} \left(c^0\right) \right)^2}\left(f''_{\mathrm{act}}\left(c^0\right) + \frac{M}{2 \pi R_0^2} f'''_{\mathrm{act}}\left(c^0\right)\right) + \dfrac{\chi_u f'''_{\mathrm{und}}\left(0\right)}{3 a c^0 R_0 f'_{\mathrm{act}}\left(c^0\right)}.
		\end{equation*}
	\end{lemma}
	
	\begin{proof}
	    Differenciating twice with respect to $s$ \cref{eq:undercooling_bifurcation_ds_F1} and using the fact that for all $\theta \in \left(-\pi,\pi\right]$ we have $\partial_s z\left(0,\theta\right) = - a c^0 R_0 \cos \theta$ and the result of the previous lemma, it follows that, for all $\theta \in \left(-\pi,\pi\right]$:
		\begin{align*}
			0 = & \gamma \partial_{sss}^3 \kappa\left(0,\theta \right) + \chi'''_c\left(0\right) f_{\mathrm{act}} \left( c^0 \right) - 3 a c^0 R_0 \chi''_c\left(0\right) f'_{\mathrm{act}} \left( c^0 \right) \cos \theta\\
			& - 3 \chi_c^* a c^0 R_0 \cos \theta \partial_{ss}^2 z\left(0,\theta\right)  f''_{\mathrm{act}}\left(c^0\right)  \\
			& - \chi_c^* f'''_{\mathrm{act}}\left(c^0\right) \left(a c^0 R_0 \cos \theta\right)^3  + \chi_c^* f'_{\mathrm{act}} \left( c^0\right) \partial^3_{sss} z\left(0,\theta\right) + \chi_u f'''_{\mathrm{und}} \left(0\right) \cos^3 \theta \\
			& + 3 \chi_u f''_{\mathrm{und}} \left(0\right) \cos^2\theta V''\left(0\right) + \chi_u f'_{\mathrm{und}} \left(0 \right) \left( V'''\left(0\right) \cos \theta + 3 \partial_{ss}^2 n_1\left(0,\theta\right) \right) \\
			& + V'''\left(0\right) R_0 \cos \theta +3 \partial_{ss}^2 \rho\left(0,\theta\right) \cos \theta - p_1'''\left(0\right).
		\end{align*}
		
		We multiply the previous equality by $\cos \theta$ and integrate in $\theta$ over $\left(-\pi,\pi\right)$, we thus have:
		\begin{align*}
			0 = & \gamma \int_{-\pi}^{\pi}  \partial_{sss}^3 \kappa\left(0,\theta \right) \cos \theta \dx \theta - 3 \pi a c^0 R_0 \chi''_c\left(0\right) f'_{\mathrm{act}} \left( c^0 \right) \\
			& - 3 \chi_c^* a c^0 R_0 f''_{\mathrm{act}}\left(c^0\right)  \int_{-\pi}^{\pi}  \partial_{ss}^2 z\left(0,\theta\right) \cos^2 \theta  \dx \theta  \\
			& - \frac{3}{4} \chi_c^* f'''_{\mathrm{act}}\left(c^0\right) \pi \left(a c^0 R_0 \right)^3  + \chi_c^* f'_{\mathrm{act}} \left( c^0\right) \int_{-\pi}^{\pi} \partial^3_{sss} z\left(0,\theta\right) \cos \theta \dx \theta \\
			& + \frac{3}{4}\chi_u f'''_{\mathrm{und}} \left(0\right) \pi + \chi_u f'_{\mathrm{und}} \left(0 \right) V'''\left(0\right) \pi + V'''\left(0\right) R_0 \pi.
		\end{align*}
		
		From \citet{alazard_traveling_2022} (appendix C), we have for all $\theta \in \left(-\pi,\pi\right]$:
		\begin{equation*}
			\int_{-\pi}^{\pi} \partial_{ss}^2 z\left(0,\theta\right) \cos^2 \theta \dx \theta = \dfrac{a^2 M}{2},
		\end{equation*}
		\begin{equation*}
			\int_{-\pi}^{\pi} \partial_{sss}^3 z\left(0,\theta\right) \cos \theta \dx \theta = -\frac{aM}{R_0}V'''\left(0\right) - \frac{2M}{\pi R_0^3} \int_{-\pi}^{\pi} \partial_{sss}^3 \rho\left(0,\theta\right) \cos \theta \dx \theta,
		\end{equation*}
		and
		\begin{equation*}
			\int_{-\pi}^{\pi} \partial_{sss}^3 \kappa\left(0,\theta\right) \cos \theta \dx \theta = 0.
		\end{equation*}
		
		It follows that:
		\begin{align*}
			0 & = - 3 \pi a c^0 R_0 \chi''_c\left(0\right) f'_{\mathrm{act}} \left( c^0 \right) - \frac{3}{2} \chi_c^* a^3 c^0 R_0  f''_{\mathrm{act}}\left(c^0\right) M \\
			&\quad - \frac{1}{R_0} a M \chi_c^* f'_{\mathrm{act}}\left(c^0\right) V'''\left(0\right)  - \frac{3}{4} \pi \chi_c^* f'''_{\mathrm{act}}\left(c^0\right) \left(a c^0 R_0 \right)^3 \\
			& \quad + \chi_u f'''_{\mathrm{und}} \left(0\right) \pi + \chi_u f'_{\mathrm{und}} \left(0 \right) V'''\left(0\right) \pi + V'''\left(0\right) R_0 \pi, \\
			& =  -3 \pi a c^0 R_0 f'_{\mathrm{act}}\left(c^0\right) \chi''_c\left(0\right) - \dfrac{3 \chi_c^* a^3 M^2}{2 \pi R_0}\left( f''_{\mathrm{act}}\left(c^0\right) + \frac{M}{2 \pi R_0^2}f'''_{\mathrm{act}}\left(c^0\right) \right) + \frac{3}{4}\pi \chi_u f'''_{\mathrm{und}}\left(0\right).
		\end{align*}
		
		It leads to:
		\begin{equation*}
			\chi''_c\left(0\right) = \dfrac{- \chi_c^* a^2 M}{2 \pi f'_{\mathrm{act}}\left(c^0\right)}\left(f''_{\mathrm{act}}\left(c^0\right) + \dfrac{M}{2 \pi R_0^2} f'''_{\mathrm{act}}\left(c^0\right)\right) + \dfrac{\chi_u f'''_{\mathrm{und}}\left(0\right)}{4 a c^0 R_0 f'_{\mathrm{act}}\left(c^0\right)}.
		\end{equation*}
		We conclude using the expressions of $\chi_c^*$ and of $c^0$. 
	\end{proof}
	
	\section{Conclusion and perspectives}
	We have proved that, despite simplifying the principles governing cell motility, the model is relevant for modeling it. It captures the phenomena of cell migration and polarization depending on the parameters. We observe that the cell membrane has an impact on the shape of the traveling wave domain through its action on the curvature (see \cref{eq:tw_curvature_equation}). 

    Furthermore, by comparing the bifurcation result obtained with the case where $\chi_u = 0$, we can conclude that the effect of the membrane is stabilizing. Indeed, the threshold under which the stationary state is stable is greater than in the case where the membrane effect is neglected.
	
	In a future work extending \citet{lavi_implicit_2023}, we plan to investigate different approaches to discretize the model with implicit temporal discretization of curvature and the undercooling effect. Finite element methods are used for the spatial discretizations. The term added by the undercooling effect requires precautions to be taken regarding the functional spaces chosen to discretize the model. The velocity trace on the domain boundary must always be well defined.

	\bibliographystyle{apalike}

\end{document}